%% file: main.tex
\documentclass[reqno]{amsart}
\usepackage{xcolor,leftindex}
\input{header}

\begin{document}

\title{Stochastic Nash evolution}

\author{Dominik Inauen}

\address{Institut f\"{u}r mathematik, Universit\"{a}t Leipzig, D-04109, Leipzig, Germany.}
\email{dominik.inauen@math.uni-leipzig.de}
\author{Govind Menon}

\address{Division of Applied Mathematics, Brown University, 182 George St., Providence, RI 02912.}
\email{govind\_menon@brown.edu}

\thanks{This work was supported by the National Science Foundation (DMS 171487 and 2107205), the Simons Foundation (Award 561041) and the Charles Simonyi Foundation. The authors also express their gratitude to the School of Mathematics at the Institute for Advanced Study, Princeton and the Max Planck Institute for Mathematics in the Sciences, Leipzig for partial support during the completion of this work.}


\begin{abstract}
This paper introduces a probabilistic formulation for the isometric embedding of a Riemannian manifold $(M^n,g)$ into Euclidean space $\R^q$. Given $\alpha \in ]\tfrac{1}{2},1]$, we show that a $C^{1,\alpha}$ embedding $u: M \to \R^q$ is isometric if and only if the intrinsic and extrinsic constructions of Brownian motion on $u(M)\subset \R^q$ yield processes with the same law. The equivalence is first established for smooth embeddings; this is followed by a renormalization procedure for $C^{1,\alpha}$ embeddings. In particular, we also construct extrinsic Brownian motion when $g \in C^2$ and $u$ is a $C^{1,\alpha}$ isometric embedding.

This formulation is based on a gedanken experiment that relates the intrinsic and extrinsic constructions of Brownian motion on an embedded manifold to the measurement of geodesic distance by observers in distinct frames of reference. This viewpoint provides a thermodynamic formalism for the isometric embedding problem that is suited to applications in geometric deep learning, stochastic optimization and turbulence.
\end{abstract}

\maketitle

\input{intro}
\input{proofs}

\input{thermo}

\bibliographystyle{siam}
\bibliography{im}
\end{document}

%% file: header.tex
\theoremstyle{plain}
\begingroup
 
\newtheorem{theorem}{Theorem} 
\newtheorem{corollary}{Corollary}
\newtheorem{lemma}{Lemma}
\endgroup
\newcommand{\nwc}{\newcommand}
\nwc{\qref}[1]{(\ref{#1})}
\nwc{\cadlag}{c\`{a}dl\`{a}g}
\nwc{\la}{\label}
\nwc{\nn}{\nonumber}
\nwc{\Z}{\mathbb{Z}}
\nwc{\C}{\mathbb{C}}
\nwc{\T}{\mathbb{T}}
\nwc{\E}{\mathbb{E}}
\nwc{\R}{\mathbb{R}}
\nwc{\N}{\mathbb{N}}
\nwc{\Rn}{\mathbb{R}^n}
\nwc{\PP}{\mathcal{P}}
\nwc{\M}{\mathcal{M}}
\nwc{\Ito}{It\^{o}}
\nwc{\DiffM}{\mathrm{Diff}(M)}

\nwc{\law}{\stackrel{\mathcal{L}}{\rightarrow}}
\nwc{\eqd}{\stackrel{\mathcal{L}}{=}}
\nwc{\vp}{\varphi}
\nwc{\veps}{\varepsilon}
\nwc{\eps}{\veps}
\nwc{\dnto}{\downarrow}
\nwc{\nsup}{^{(n)}}
\nwc{\ksup}{^{(k)}}
\nwc{\jsup}{^{(j)}}
\nwc{\nksup}{^{(n_k)}}
\nwc{\inv}{^{-1}}
\nwc{\argmin}{\mathrm{argmin}}
\nwc{\argmax}{\mathrm{argmax}}
\nwc{\Tr}{\mathrm{Tr}}
\nwc{\Id}{\mathrm{Id}}
\nwc{\Pn}{\mathbb{P}(n)}
\nwc{\PnR}{\mathbb{P}(n;\mathbb{R})}
\nwc{\PnC}{\mathbb{P}(n;\mathbb{C})}
\nwc{\Hn}{\mathbb{H}(n)}
\nwc{\HnR}{\mathbb{H}(n;\mathbb{R})}
\nwc{\HnC}{\mathbb{H}(n;\mathbb{C})}
\nwc{\An}{\mathbb{A}(n)}
\nwc{\AnR}{\mathbb{A}(n;\mathbb{R})}
\nwc{\AnC}{\mathbb{A}(n;\mathbb{C})}
\nwc{\Mn}{\mathbb{M}(n)}
\nwc{\feasible}{\mathcal{F}}
\nwc{\GLn}{GL(n)}
\nwc{\GLnC}{GL(n;\mathbb{C})}
\nwc{\GLnR}{GL(n;\mathbb{R})}
\nwc{\Poincare}{Poincar\'{e}}
\nwc{\Un}{U(n)}
\nwc{\Sn}{\mathbb{S}^{n}}
\nwc{\Tn}{\mathbb{T}^n}
\nwc{\ev}{\mathrm{ev}}
\nwc{\Diff}{\mathrm{Diff}}

\nwc{\Levy}{L\'{e}vy}
\nwc{\HC}{Harish-Chandra}
\nwc{\BW}{Bures-Wasserstein}
\nwc{\CH}{Cartan-Hadamard}
\nwc{\Hess}{\mathrm{Hess}}
\nwc{\Gunther}{\mathrm{G\"{u}nther}}

\theoremstyle{definition}

\newtheorem{conjecture}[theorem]{Conjecture}
\newtheorem{remark}{Remark} 
\theoremstyle{remark}
 
\numberwithin{equation}{section}
\numberwithin{figure}{section}

\nwc{\Sz}{Sz\'{e}kelyhidi}
\nwc{\DiffMg}{\DiffM}

%% file: intro.tex
\section{Introduction}
This work presents a probabilistic characterization of  the isometric embedding problem for Riemannian manifolds based on the equivalence between the intrinsic and extrinsic constructions of Brownian motion on the manifold. We introduce these problems, state and prove our results, and then discuss their scientific context.  

\subsection{The isometric embedding problem for Riemannian manifolds}
\label{subsec:bg}
Assume given a smooth $n$-dimensional closed differentiable manifold $M$ equipped with a metric $g$ and let $(\R^q,e)$ denote $q$-dimensional Euclidean space with the identity metric. An immersion $u: M \to \R^q$ is isometric if $u^\sharp e=g$, where $u^\sharp e$ denotes the pullback metric. In coordinates, this is the system of partial differential equations
\begin{equation}
\label{eq:embed1}
\sum_{a=1}^q\partial_{i}u^a \partial_j u^a = g_{ij}.
\end{equation}
Here $a$ indexes coordinates in $\R^q$ and $1\leq i,j \leq n$ index coordinates on $M$. An embedding is an immersion that is one to one. The isometric embedding and immersion problems are closely related: the central difficulty in both problems is the analysis of equation~\eqref{eq:embed1}. The results in this paper have a natural extension to isometric immersions, but we focus on the embedding problem to be concrete. Similarly, we have assumed that $M$ is closed only to focus on the essential ideas.

Nash's work on the isometric embedding problem has determined the development of the subject since the 1950s. In 1954, he proved the existence of surprising $C^1$ isometric embeddings assuming that $g \in C^0$ and that there are no topological obstructions that prevent smooth embeddings of $M$ into $\R^q$ (this is always true when $q\geq 2n$)~\cite{Nash1}.  In 1956, he established the existence of $C^\infty$ isometric embeddings when $g\in C^\infty$ and $q$ is large enough~\cite{Nash2}. These  theorems have played a seminal role in several areas of mathematics as discussed in the surveys~\cite{DeS3,Gromov-AMS,Hamilton-nash-moser,Klainerman-AMS}. 

\subsection{Brownian motion on Riemannian manifolds}
\label{subsec:BM}
Our sources for stochastic analysis are the monographs~\cite{Hsu,Ikeda,Stroock-book}. 

Let $\triangle_g$ denote the Laplace-Beltrami operator on $(M,g)$. Its action on a smooth function $f$ is given in coordinates by
\begin{equation}
\label{eq:laplace}
\triangle_g f = \frac{1}{\sqrt{|g|}} \partial_i \left( \sqrt{|g|} g^{ij} \partial_j f \right),
\end{equation}
where  $|g|$ denotes the determinant of $g_{ij}$ and $g^{ij}$ denotes its inverse. 

Brownian motion on $(M,g)$ is a diffusion on $M$ whose generator is $\tfrac{1}{2}\triangle_g$. We begin by reviewing the classical constructions of Brownian motion using Stratonovich SDE (see for example~\cite[\S 3.2]{Hsu} and~\cite[\S V.4]{Ikeda}). These serve as a foundation for our work in the low regularity regime.

\subsubsection{Intrinsic Brownian motion}
The standard intrinsic construction of Brownian motion on $(M,g)$, introduced by Eells, Elworthy and Malliavin, proceeds as follows. Consider the bundle of orthonormal frames $O(M)$ for $(M,g)$, let $\pi: O(M)\to M$ denote the canonical projection and let $\{H_j\}_{j=1}^n$ denote the fundamental horizontal vector fields. Let $\{W^j\}_{j=1}^n$ denote $n$ independent standard Wiener processes and consider the Stratonovich SDE on $O(M)$
\begin{equation}
\label{eq:intrinsic}
dU_t = \sum_{j=1}^n \left. H_j\right|_{U_t} \circ dW^j_t.
\end{equation}
The generator for the diffusion $U_t$ is one half the Bochner Laplacian on $O(M)$,
\begin{equation}
    \label{eq:bochner}
    \triangle_{O(M)} = \sum_{j=1}^n H_j^2.
\end{equation}
The projection $X_t=\pi(U_t)$ is a Brownian motion on $(M,g)$ starting at $x=\pi(U_0)$.

While the Bochner Laplacian may be written as a sum of squares, there is no such canonical decomposition for the Laplacian $\triangle_g$ on $(M,g)$. Thus, $X_t$ itself does not admit such a natural formulation as an SDE. However, in any coordinate patch, the Stratonovich SDE \eqref{eq:intrinsic} implies the \Ito\/ SDE
\begin{equation}
    \label{eq:intrinsic2}
    dX^i_t = \sigma^i_j(X_t) dW_t^j -b^i \, dt,
\end{equation}
where the drift and covariance of the noise are given by (see~\cite[Ex. 3.3.5]{Hsu})
\begin{equation}
    \label{eq:intrinsic3}
    b^i = \frac{1}{\sqrt{|g|}} \frac{\partial \sqrt{|g|} g^{ij}}{\partial x^j},  \quad \sigma^{i}_k\sigma^{j}_k = g^{ij}.
\end{equation}
We may now use \Ito's formula to see that the generator of $X_t$ is $\tfrac{1}{2}\triangle_g$.

\subsubsection{Extrinsic Brownian motion}
Assume that $u: M\to \R^q$ is a smooth isometric embedding and denote by $\Sigma=u(M)$ the image of $M$ in $\R^q$. We may project Brownian motion on $\R^q$ onto $\Sigma$  as follows. Let $\{B^a\}_{a=1}^q$ denote $q$ independent standard Wiener processes, let $\{P_a\}_{a=1}^q$ denote the orthonormal projection of the standard basis vectors $e_a \in \R^q$ onto the tangent space $T\Sigma$, let $x\in M$, and consider the Stratonovich SDE
\begin{equation}
    \label{e:extrinsicBM}
    dZ_t = \sum_{a=1}^q P_a \circ dB^a_t, \quad Z_0=u(x).
\end{equation}
The solution to this SDE also corresponds to a Brownian motion on $(M,g)$ beginning at the point $x$ (see~\cite[Eq. 3.2.6, p.87]{Hsu}). 

\subsection{Overview of results} 
\label{subsec:main-thm}
Our main insight is that the equivalence of these constructions of Brownian motion characterizes the isometric embedding problem for $(M,g)$. This idea is natural for smooth embeddings. Our main results establish this characterization in the regime when both the metric $g$ and the embedding $u$ have low regularity. We begin with some remarks that explain our viewpoint.

The construction of Brownian motion on $(M,g)$ using equation~\eqref{eq:intrinsic} is appealing because it is an extension of the intrinsic notion of parallel transport to  the stochastic setting. The ordinary differential equations that describe parallel transport have been replaced by Stratonovich SDE in equation~\eqref{eq:intrinsic}. The use of the bundle of frames allows us to think of Brownian motion on $(M,g)$ as the limit of small, independent, isotropic random parallel translations. It does not require embeddability of the manifold and it requires exactly $n$ driving Brownian motions. 

The extrinsic construction reduces Brownian motion on $(M,g)$ to the familiar construction of Brownian motion on $\R^q$. This simplifies many calculations and is amenable to numerical simulations. However, it is (at least at first sight) less satisfactory than the intrinsic construction. The flaws in this approach include the assumption of the existence of a smooth embedding and the fact that the number of driving Brownian motions is $q \geq n +1$. 

The isometric embedding problem itself has been a subject of renewed study in the past ten years following the creation of the embedding-turbulence analogy by De Lellis and \Sz\/ (see~\cite{DeS3} for a survey). An important area of enquiry has been the critical exponent for $C^{1,\alpha}$ embeddings that separate rigidity and flexibility, in analogy with the Onsager conjecture for the Euler equations~\cite{CI,Cao,CDS,DI,DIS,Isett} (see also \cite{CIH, CSMA,Lewicka-Pakzad} for closely related results for the Monge--Amp\`ere equation). The embedding-turbulence analogy, as well as the importance of the isometric embedding problem for several techniques in machine learning, has motivated us to revisit Nash's work using probabilistic techniques in order to align  scientific applications with mathematical foundations. 

Theorem~\ref{t:main}--Theorem~\ref{t:renorm} provide a rigorous foundation for the information theoretic interpretation of isometric embedding introduced in~\cite{GM-gsi}. While the conceptual foundation for the work in this paper is similar, the techniques used here reveal the probabilistic nature of the isometric embedding problem in its simplest form. These theorems should be seen as the first steps towards a rigorous, and minimal, thermodynamic formalism for the isometric embedding problem. They have already been used for the design of algorithms for Gibbs sampling and optimization, for the construction of models in random matrix theory, as well as to shed new light on geometric deep learning. These applications are discussed in  Section~\ref{sec:thermo}.

\section{Statement of results}
\subsection{Overview}
We study carefully the manner in which we may recover geometric information from strictly probabilistic foundations. Our main result, Theorem~\ref{t:renorm} below, provides an equivalence between the isometric embedding problem and the intrinsic and extrinsic constructions of Brownian motion for $C^2$ metrics and $C^{1,\alpha}$ embeddings when $\alpha \in ]\tfrac{1}{2},1]$. It also provides a pathwise construction of extrinsic Brownian motion on a $C^{1,\alpha}$ embedding for $\alpha \in ]\tfrac{1}{2},1]$. In turn, the sample path properties of this extrinsic Brownian motion shed new light on the role of (renormalized) mean curvature for $C^{1,\alpha}$ isometric embeddings.


%

The equivalence is natural in the $C^\infty$ regime (see Theorem~\ref{t:main}). However, it is necessary to combine
techniques from PDE theory (especially elliptic regularity theory and $h$-principles) with stochastic analysis to extend this equivalence to the low regularity regimes. It is  also necessary to construct intrinsic Brownian motion for metrics with low regularity before studying embeddings. We do this in Theorem~\ref{t:intrinsicBMlowreg}, replacing the Eells-Elworthy-Malliavin construction with a martingale problem for $C^2$ metrics. Then for $g\in C^2$, we extend Theorem~\ref{t:main} into the regime of $C^{1,\alpha}$ embeddings, $\alpha \in ]\tfrac{1}{2},1]$, using a natural geometric renormalization procedure. 

The relationship between curvature and Brownian motion plays an important role in stochastic analysis. We develop this idea for $C^{1,\alpha}$ isometric embeddings as follows. We first recall that mean curvature may be understood probabilistically  for $C^\infty$ embeddings (roughly it is is the `constraint force' that keeps Brownian motion on a curved manifold). Conjecture~\ref{conj:smg} (which holds for $C^{1,2/3+}$ convex surfaces) then extends the notion of mean curvature to $C^{1,\alpha}$ embeddings.

Conceptually, these theorems formalize the process of measurement of length by intrinsic and extrinsic observers. We refer to the probability spaces on which the Wiener processes $W_t\in \R^n$ and $B_t \in \R^q$ are defined as the intrinsic and extrinsic probability spaces respectively. We may use the probabilities of bridges connecting points $x$ and $y$ in $M$ to define statistical estimators for the geodesic distance between these points. Thus, the laws of intrinsic and extrinsic Brownian motion provide a model for the measurement of length by intrinsic and extrinsic observers.






\subsection{Smooth embeddings}
\label{sec:statement}
We always assume that $M$ is a closed, connected manifold with a $C^\infty$ differentiable structure. In this section, we also assume that both the metric $g$ and embedding $u$ are $C^\infty$. 

Assume given an embedding $u:M \to \R^q$ (not necessarily isometric).  Fix $x \in M$, let $Y_t$ denote the process $u(X_t)$ obtained from equation~\eqref{eq:intrinsic} and assume $Z_t$ is as above in \eqref{e:extrinsicBM}. Both these processes are curves on $\Sigma$ that originate at $u(x)$.

\begin{theorem}\label{t:main}
Assume $u: M \to \R^q$ is a $C^\infty$ embedding and $g$ is a $C^\infty$ metric. The processes $Y_t$ and $Z_t$ have the same law for each $x \in M$ if and only if the embedding $u$ is isometric.
\end{theorem}
Theorem~\ref{t:main} provides an equivalence between the isometric embedding problem and the construction of Brownian motion on a manifold $(M,g)$ for $C^\infty$ embeddings. The main ideas are as follows. Of course, we expect to have agreement between the laws of $Y_t$ and $Z_t$ when $u$ is isometric. But this requires a proof since the processes $U_t$ and $Z_t$ are constructed on different probability spaces, their sample paths lie in $O(M)$ and $\Sigma\subset \R^q$ respectively, and the regularity assumptions on $u$ and $g$ play different roles in these constructions. The converse provides a strictly probabilistic formulation of isometric embedding: equality in law of the processes $Y_t$ and $Z_t$ determines isometry of a (sufficiently smooth) embedding.

Once the converse has been conceived, it is easy to prove. If the processes $Y_t$ and $Z_t$ have the same law they yield the same heat kernels $k(t, x,y)$ for pairs of points $x,y\in M$ and $t>0$. This allows us to determine the geodesic distance $d_g(x,y)$ between  pairs of points $x,y \in M$, and thus the metric itself, using  Varadhan's lemma (equation (1.5) in \cite{Varadhan})
\begin{equation}    
\label{eq:varadhan}
d_g^2(x,y) = -2 \lim_{t\to 0} t \log k(t,x,y).
\end{equation}
This argument is robust, but acquires several subtle new features, in the low regularity regimes for $g$ and $u$.

\subsection{Extrinsic Brownian motion and mean curvature}
\label{sec:mc}
 We begin by recalling a classical relationship between extrinsic Brownian motion and mean curvature for $C^\infty$ isometric embeddings. See~\cite[\S 2.6]{Hsu},~\cite{Lewis}, and ~\cite[Thm 4.42]{Stroock-book} for variants of Theorem~\ref{t:curvature} below.


Let $II(z)$ denote the second fundamental form of the embedding for $z\in \Sigma$ and define the mean curvature $H(z)$ as the trace of $II(z)$. The correction term in the \Ito\/-Stratonovich conversion for equation \eqref{e:extrinsicBM} is as follows.
\begin{theorem}[Lewis~\cite{Lewis}]
\label{t:curvature}
Assume $u: M \to \R^q$ is $C^{\infty}$. The \Ito\/ form of the Stratonovich SDE \eqref{e:extrinsicBM} is
\begin{equation}
    \label{eq:curvature}
    dZ_t = \sum_{a=1}^q P_a dB^a_t + \frac{1}{2}H(Z_t) \,dt. 
\end{equation}
\end{theorem}

Our main interest lies in the extension of Theorem~\ref{t:curvature} to $C^{1,\alpha}$ embeddings. In this setting, the mean curvature is {\em not\/} defined. However, the existence of extrinsic Brownian motion suggests a replacement of the stochastic differential $\frac{1}{2}H(Z_t) \,dt$ by a geometric analog of the local time (see Conjecture~\ref{conj:smg} below). 

This version of Theorem~\ref{t:curvature} has proved to be of value in the development of algorithms, matrix models and geometric stochastic flows. Typically, in each of these applications we compute the mean curvature of a group orbit explicitly and use it as the basis for a random matrix model or numerical scheme~\cite{HIM23,MY1,MY2,Yu1}. For these reasons, we include a proof of Theorem~\ref{t:curvature} in this paper.



\subsection{Background on critical exponents in the $C^{1,\alpha}$ regime}
\label{subsec:background}
We briefly review the main ideas in Nash's 1954 proof for the existence of $C^1$ embeddings and the manner in which these ideas persist in the embedding-turbulence analogy and related critical exponent problems.

Let us simplify matters by assuming that the topology of $M$ is such that it admits a $C^\infty$ embedding into $q=n+2$. By rescaling, we may then assume that $M$ admits a $C^\infty$ {\em short\/} embedding $u_0$. This is a {\em subsolution\/} to equation~\eqref{eq:embed1}, i.e $u_0^\sharp e < g$ in the order on symmetric quadratic forms. Nash established the existence of $C^1$ isometric embeddings by constructing a sequence $\{u_k\}_{k=1}^\infty$ of $C^\infty$ subsolutions such that the pullback metrics increase monotonically to $g$,
\[ u_0^\sharp e < \ldots u_k^\sharp e < u^\sharp_{k+1}e < \ldots < g. \]
This family is constructed through a careful addition of small, high-frequency oscillations. A fundamental step is the {\em Nash lemma\/}, decomposing a $C^\infty$ metric $h$ into a finite `sum of squares' of the form
\[ h = \sum_{l=1}^N a_l^2 \, d\psi_l \otimes d\psi_l.\]
Here $a_l$ and $\psi_l$ are $C^\infty$ functions on $M$. The integer $N$ and the set $\{\psi_l\}_{l=1}^N$ are determined by a simplicial decomposition of the manifold. Given $h$, the Nash lemma provides the coefficients $\{a_l\}_{l=1}^N$~\cite[Lemma 1]{Nash1}. 

This lemma is used in the following way. Nash's iteration consists of an outer loop (`stages') indexed by $k =0,1,2,\ldots$ and an inner loop (`steps') indexed by $l=1,\ldots, N$. At the $k$-th stage, the residual $h=g-u_k^\sharp e$ is decomposed using the Nash lemma. Then $u_k$ is modified in $N$ steps, with the $l$-th step correcting the pullback metric by (approximately) $a_l^2 d\psi_l \otimes d\psi_l$ through the addition of small, high-frequency, normal and binormal fluctuations. 

The embedding-turbulence analogy relies on an analogous, and rather subtle, notion of subsolutions for the Euler equations which involves replacing the metric $g$ with the Reynolds stress~\cite{DeS2}. This is followed by the construction of a sequence of subsolutions that converge `upwards' to a weak solution to the Euler equations.  H\"{o}lder regularity is obtained by interweaving the addition of oscillations with a smoothing step. The estimates are challenging and the resolution of the Onsager conjecture by Isett~\cite{Isett} builds on several patient improvements on the first such scheme in~\cite{DeS2}. The improvements that provide 
H\"{o}lder regularity for both the Euler equations and isometric embedding problem rely on commutator estimates and smoothing~\cite{CI,Cao,CDS,DI,DIS}. Other $h$-principles that follow~\cite{DeS2} rely on similar decompositions; a key step in the construction is an analog of the Nash lemma.

Theorem~\ref{t:main} allows a different approach to these questions. By viewing the equivalence of two constructions of Brownian motion as the main characteristic of isometric embedding, we seek a critical exponent $\alpha_*$ such that (a modification of) Theorem~\ref{t:main} holds for $C^{1,\alpha}$ embeddings with $\alpha>\alpha_*$. In the spirit of Stroock and Varadhan's introduction of the martingale problem for diffusions in $\R^d$~\cite{StroockVaradhan}, we also ask if Theorem~\ref{t:main} may be used to define a martingale characterization of the isometric embedding problem. This requires an extension of both theorems into the low regularity regime for $u$ and $g$ and a switch from the use of stochastic differential equations to martingale problems. Let us now consider the regularity assumptions that underlie Theorem~\ref{t:main} and Theorem~\ref{t:curvature} more carefully.

\subsection{Construction of intrinsic Brownian motion for $g\in C^2$} 
The construction of the process $U_t$ using the SDE~\eqref{eq:intrinsic} relies only on the smoothness of the differential structure of $M$ (which we have assumed $C^\infty$), as well as the smoothness of the metric $g$. The standard monographs in the area assume $g\in C^\infty$ for simplicity, although a more careful argument shows that $g \in C^3$ suffices. In Theorem~\ref{t:intrinsicBMlowreg} below, we construct intrinsic Brownian motion for $g\in C^2$ as the solution to a martingale problem. Let us explain the issue.

When $g \in C^2$, equation \eqref{eq:intrinsic} is well defined in every coordinate chart, but the coefficients in the equation are not regular enough to guarantee strong uniqueness to the corresponding It\^o SDE in a given chart. This prevents us from constructing a solution to \eqref{eq:intrinsic} by patching together the strong solutions on different coordinate charts (see e.g. the proof of Theorem V.1.1. in \cite{Ikeda}). 

Instead, we construct intrinsic Brownian motion on $(M,g)$ by approximation. In what follows, we consider a sequence of $C^\infty$ Riemannian metrics denoted $g_k$. For any $x\in M$, let $X^{k,x}$ be the intrinsic Brownian motion on $(M,g_k)$ starting at $x$ (see subsection \ref{subsec:BM}). For any $k\geq 1$ we consider the law of the process $X^{k,x}$ as a probability measure $P^{k,x}$ on the path space $C([0,+\infty[, M)$. We then have 

\begin{theorem}\label{t:intrinsicBMlowreg} 
Assume $g$ is a $C^2$ metric and assume that $\{g_k\}_{k=1}^\infty$ is a sequence of $C^\infty$ metrics such that $\lim_{k\to \infty} \|g_k -g\|_{C^1}=0$. 
The sequence $\{P^{k,x}\}_{k\geq 1}$ converges weakly to a solution $P^{x}$ of the martingale problem for $\frac{1}{2}\Delta_{g}$ starting at $x$. Further, this solution to the martingale problem is unique.
\end{theorem}
As a consequence there exists a probability space $(\Omega,\mathcal F, \mathbb P)$, a filtration $\mathcal F_t$ and an adapted process $\{X_t\}_{t\geq 0}:\Omega \to M$ such that $X_0 =x$ almost surely and for every $f\in C^{2}(M)$ it holds that 
\begin{equation}\label{e:martingaleprop}
f(X_t)- f(x) - \frac{1}{2}\int_0^{t}\Delta_{g}f(X_s)\,ds
\end{equation}
is a martingale with respect to $\mathbb{P}$ and $\mathcal F_t$, i.e. $\{X_t\}_{t\geq 0}$ is intrinsic Brownian motion starting at $x$. Moreover, we find 

\begin{corollary}\label{c:markov1} For any $x\in M$, the process $\{X^x_t\}_{t\geq 0}$ is a homogenous Markov process with transition density given by the heat kernel $k_g$.
\end{corollary}

For the definition of the heat kernel $k_g$ see Subsection~\ref{s:heateq}. The latter corollary is a direct consequence of the Markov property of unique solutions to a martingale problem and the $C^2$ smoothness of the solution of the heat equation on $(M,g)$. 

\begin{remark}
\label{rem:smooth-metric} We remark that the statement of the martingale problem for $\Delta_g $ makes sense if $g\in C^1$. However in the proof of Theorem \ref{t:intrinsicBMlowreg} we need the stronger $g\in C^2$. 
\end{remark}

\subsection{Renormalization and the main theorem} 
The process $Z_t$ is defined using the Stratonovich SDE~\eqref{e:extrinsicBM}. Standard SDE theory shows that $u\in C^3$ is the natural hypothesis for the existence of strong solutions to~\eqref{e:extrinsicBM}. However, the natural minimum regularity threshold for isometric maps is when $u$ is Lipschitz. In particular, the intrinsic process $Y_t=u(X_t)$ has a natural meaning for isometric embeddings when $u$ is Lipschitz and $g \in C^2$. Indeed, the existence of solutions to the martingale problem for $X_t$ requires only that $g\in C^2$, so that $Y_t=u(X_t)$ is defined for every measurable $u$.  

Thus, we have a considerable difference between the smoothness assumptions on $g$ and $u$ required to construct the processes $X_t$, $Y_t$ and $Z_t$. In Theorem~\ref{t:renorm} below, we resolve this mismatch for $C^{1,\alpha}$ embeddings, $\alpha>\tfrac{1}{2}$. 


We are guided by a geometric interpretation of Tanaka's formula. 
The map $x \mapsto |x|$ is an isometric folding of the line onto the half-line. Let $X_t$ be Brownian motion on $\R$, set $Y_t=|X_t|$, and let $L_t$ be the local time of $X_t$ at zero. Then Tanaka's formula tells us that $Y_t$ is a semimartingale and 
\begin{equation}
    \label{e:tanaka}
    dY_t = \mathrm{sgn}(X_t) \, dX_t + dL_t.
\end{equation}
Further, by following the process $(X_t,Y_t)$ on its  graph in $\R^2$, we see that the constraint of mean curvature has been replaced by random normal kicks at $(0,0)$ determined by the local time $L_t$. 

This interpretation of Tanaka's fromula suggests the following renormalization procedure. 
For small $\varepsilon>0$ let $u^\varepsilon: M \to \R^q$ denote a mollified embedding $u^\varepsilon= u* \varphi_\varepsilon$ (see Lemma~\ref{lem:smoothing}). Let $g^\varepsilon$ denote the pullback metric $(u^\varepsilon)^\sharp e$. Fix $x \in M$ and let $X^\varepsilon_t$ denote intrinsic Brownian motion on $(M,g^\varepsilon)$ with $X^\varepsilon_0=x$. Let $Y^\varepsilon_t = u^\varepsilon(X^\varepsilon_t)$ denote its image on $u^\varepsilon(M)$. Let us denote the extrinsic Brownian motion beginning at $u^\varepsilon(x)$ by $Z^\varepsilon_t$. Since $u^\varepsilon$ is a $C^\infty$ embedding and $Y_t^\varepsilon$ is constructed with the pullback metric $(u^\varepsilon)^\sharp e$, Theorem~\ref{t:main} implies that the processes $Y_t^\varepsilon$ and $Z_t^\varepsilon$ have the same law, absolutely continuous with respect to the Wiener measures for $W_t \in \R^n$ and $B_t \in \R^q$ respectively. We then have
\begin{theorem}
\label{t:renorm}
Assume $g \in C^2$ and $u: M \to \R^q$ is a $C^{1,\alpha}$ embedding with $\alpha > \tfrac{1}{2}$. 

\begin{enumerate}
    \item If $u$ is isometric then the processes $Z_t^\varepsilon$ converge in law to a process $Z_t$ defined on the extrinsic probability space and  $Z_t$ has the same law as $Y_t$ for every $x\in M$.
    \item Conversely, if the processes $Z_t^\varepsilon$ converge in law to a process $Z_t$ defined on the extrinsic probability space such that $Z_t$ has the same law as $Y_t$ for every $x\in M$, then $u$ is isometric.
\end{enumerate}
\end{theorem}
Under the hypothesis of part (1) of the Theorem, the law of the limit $Z_t$ is independent of the mollifier. Thus, while
the SDE~\eqref{e:extrinsicBM} does not naively hold for $C^{1,\alpha}$ embeddings, Theorem~\ref{t:main} provides a renormalized construction of extrinsic Brownian motion.
\begin{corollary}
\label{cor:equivalence}
Assume $g \in C^2$, $\alpha \in ]\tfrac{1}{2},1]$ and $u: M \to \R^q$ is a $C^{1,\alpha}$ isometric embedding. For every $x \in M$ the process $Z_t$ is an extrinsic Brownian motion started at $u(x)$. 
\end{corollary}
\subsection{Remarks on Theorem~\ref{t:renorm}}
The proof of Theorem~\ref{t:renorm} relies on a subtle interplay between analysis and probability. We summarize these issues in the remarks below.
\begin{remark}[Commutator estimates]
The proof of Theorem~\ref{t:renorm} relies on the following structure. Since $Z_t^\varepsilon$ and $Y^\varepsilon_t$ have the same law for $\varepsilon>0$, it is only necessary to show that the law of $Y_t^\varepsilon$ converges to that of $Y_t$. The important observation is that, even though $u$ is only $C^{1,\alpha}$, the pullback metrics $g_\varepsilon$ converge in $C^{1}$ to $g$ when $u$ is isometric (see Lemma \ref{l:christoffels}). This step reduces again to a use of commutator estimates for mollifications, as in~\cite{CET,CDS}. Hence by Theorem \ref{t:intrinsicBMlowreg} the processes $X_t^{\varepsilon}$ converge in law to $X_t$, and since $u_\varepsilon$ converge to $u$ in $C^{1}$, the theorem follows. 
\end{remark}

\begin{remark}[Regularity of $u^\sharp e$ and martingale problems] For $u\in C^{1,\alpha}$ the pullback metric $u^\sharp e$ is a priori only $C^{0,\alpha}$. Consequently, neither $\triangle_{u^\sharp e}$ nor $\triangle_\Sigma$ are defined as classical operators and hence the martingale problem for $\triangle_\Sigma$ or $\triangle_{u^\sharp e}$ can not be defined. In this low regularity setting we adopt the definition that a Brownian motion on a Riemannian manifold with continuous metric is a Markov process with transition density given by the heat kernel (see Subsection \ref{s:Markovproperty}).

\end{remark}

\begin{remark}[Varadhan's lemma in the $C^{1,\alpha}$ regime]
Norris has established Varadhan's lemma for Lipschitz Riemannian manifolds~\cite{Norris},
providing an important step for the converse in Theorem~\ref{t:renorm}. However, we establish equality of metrics, not just equality of geodesic distances. The assumption that $g \in C^2$ is necessary to bootstrap from Norris' theorem to the assertion $g=u^\sharp e$; see Lemma~\ref{l:distanceiso}. Again the issue is that a priori $u^\sharp e$ is only $C^{0,\alpha}$. 
 \end{remark}
 \begin{remark}[Semimartingales and the renormalization of mean curvature]
Theorem~\ref{t:curvature} and Tanaka's formula suggest a renormalization of mean curvature for $C^{1,\alpha}$ embeddings, provided one can establish
\begin{conjecture}
\label{conj:smg}
The processes $Y_t$ and $Z_t$ are $\R^q$-valued semimartingales for $\alpha > \tfrac{1}{2}$.    \end{conjecture}
This conjecture provides a rigorous framework within which we may replace the idea of `constraint force determined by mean curvature' with `random kicks in the right direction to maintain a constraint' for $C^{1,\alpha}$ embeddings. Indeed, if $Y_t$ and $Z_t$ are semimartingales, we may write the (conjectural) Doob-Meyer decomposition of $Y_t$ and $Z_t$ as
\begin{equation}
\label{eq:smg}
dY_t = Du(X_t) dX_t + dK_t, \quad dZ_t = \sum_{a=1}^q P_a dW_t^a + dM_t.    
\end{equation} 
When $u$ is smooth we have $dK_t =\tfrac{1}{2}H(Y_t)\, dt$ and $dM_t =\tfrac{1}{2}H(Z_t)\, dt$. Thus, should such a decomposition hold, the processes $K_t$ and $M_t$ renormalize the mean curvature as a geometric analog for $C^{1,\alpha}$ isometric embeddings of the local time. 

It is an interesting fact that the decomposition~\eqref{eq:smg} holds when $u$ is a $C^{1,2/3+}$ embedding of a complete two-dimensional Riemannian manifold with positive Gaussian curvature into $\R^3$. Indeed, by results due to Pogorelov \cite{Pogorelov} and Borisov \cite{Borisov1,Borisov2,Borisov3,Borisov4} (see also \cite{CDS}), the submanifold $u(M)$ is a complete convex surface in $\mathbb R^3$. One can then invoke the Meyer-\Ito\/ decomposition of semimartingales~(see~\cite{Aleksandrov} and~\cite[Thm. 70]{Protter}). 
\end{remark}

\subsection{Organization of the rest of the paper}
 These theorems are proved in the sections that follow. Section~\ref{sec:thermo} has a different character. Here we explain the conceptual foundation for our work. In particular, we explain the manner in which our theorems provide a rigorous Bayesian foundation for the thermodynamics of isometric embedding. This viewpoint ties it to many applications, including geometric deep learning, optimization and turbulence. 


%% file: proofs.tex
\section{Proof of Theorem~\ref{t:main}}
\label{sec:proofs}
For the proof of the first part of Theorem \ref{t:main} we first establish the It\^o SDE's of the processes $f(Y_t), f(Z_t)$ for general $f\in C^\infty(\Sigma)$ in Lemma \ref{l:SDEs}. If the embedding is isometric, Lemma \ref{l:iso} implies that the drift terms of these processes agree. We then use the uniqueness of solutions to the martingale problem (see \eqref{e:martingaleproperty}) shown in Lemma \ref{l:mpuniqueness} to  deduce that the laws of $Y_t$ and $Z_t$ must agree. 

We prove the converse to Theorem~\ref{t:main} after Lemma~\ref{l:mpuniqueness} is established. We explain the manner in which Varadhan's lemma is used with some care, in order to prepare the reader for the more delicate use of similar ideas for $C^{1,\alpha}$ embeddings and Theorem~\ref{t:renorm}. 

We begin by introducing some notation. For a Riemannian manifold $(N,h)$ and a function $f\in C^{\infty}(N)$ the gradient $\nabla_h f $ is the vector field on $N$ defined through $h(\nabla_h f, X) = X(f)$ for any vector field $X\in \mathfrak{X}(N)$, $|\nabla_h f|_h$ denotes its length with respect to the metric $h$, and $\Delta_h$ denotes the Laplace--Beltrami operator on $(N,h)$. The volume element $\sqrt{\det h}$ is denoted by $\sqrt{|h|}$. In the following, we equip the submanifold $\Sigma=u(M)$ with the induced (pullback) metric $h=\iota^{\sharp}e $, where $\iota:\Sigma\hookrightarrow \R^q$ denotes the inclusion and $e$ the Euclidean metric on $\R^q$. We denote indices for coordinates on $M$ by $i,j,k$, so that $1\leq i \leq n$; we denote indices for coordinates on $\R^q$ by $a,b,c$, so that $1\leq a \leq q$ etc. We denote the coordinates on $\R^q$ by $y^a$ and we typically use the summation convention. 
 
\begin{lemma}\label{l:SDEs} For $f\in C^{\infty}(\Sigma)$ it holds
\begin{align}
  df(Y_t) &= dM^{f}_t + \frac{1}{2}\Delta_g (f\circ u) (X_t)\,dt \label{e:intrinsicSDE}\\
  df(Z_t) &= dN^f_t + \frac{1}{2} \Delta_h f(Z_t) \,dt \,,\label{e:extrinsicSDE}
  \end{align} 
  where $M^f$ and $N^{f}$ are martingales with quadratic variation 
  \begin{align*} 
    \langle M^f\rangle_t &= \int_0^{t}|\nabla_g (f\circ u)(X_s)|_g^{2}\,ds,\\
    \langle N^f\rangle_t &= \int_0^{t} |\nabla_{h}  f(Z_s)|^{2}\,ds \,. \end{align*}  
\end{lemma}

\begin{proof}
We first show \eqref{e:intrinsicSDE}. Recall that the intrinsic Brownian motion on $M$ is constructed using the horizontal Brownian motion $U_t$ on the bundle of orthonormal frames $O(M)$. The horizontal Brownian motion satisfies the SDE 
\begin{equation}\label{e:horizontal}
dU_t = H_i\vert_{U_t} \circ dW^{i}_t
\end{equation}
where the sum runs over $i=1,\ldots,n$, and $H_i$ are the fundamental horizontal vector fields. In particular, they have the property that for $f\in C^{\infty}(M)$ and $r=(x,[E_1,\ldots,E_n])\in O(M)$ it holds  
\begin{equation}\label{e:derivative}H_i|_r(f\circ \pi) = E_i(f)\,,\end{equation}
and moreover, crucially,
\begin{equation}\label{e:laplacian}\sum_{i=1}^{n} H_i\vert_r( H_i(f\circ \pi)) = \Delta_gf(x)\,,\end{equation}
where  $\pi:O(M)\to M$ denotes the projection.
Since by definition $X_t = \pi(U_t)$ and thus, for $f\in C^{\infty}(\Sigma)$, $f(Y_t) = f\circ u\circ \pi (U_t)$ we find by the conversion formula (cf. Theorem 2.3.5 in \cite{Kunita}) 
\begin{align*}
df(Y_t) &= H_i|_{U_t}(f\circ u\circ \pi)\circ dW^{i}_t  = H_i\vert_{U_t}(f\circ u\circ \pi) dW^{i}_t + \frac{1}{2}\langle H_i|_{U_\cdot}(f\circ u\circ \pi), W^{i}\rangle_t dt \\
& = H_i\vert_{U_t}(f\circ u\circ \pi) dW^{i}_t + \frac{1}{2}H_j\vert_{U_t}\left ( H_i (f\circ u\circ \pi)\right )d\langle W^{j}, W^{i}\rangle_t  \\
&= E_{i,t}(f\circ u)\,dW^{i}_t + \frac{1}{2}\Delta_g (f\circ u)(X_t)\,dt
\end{align*} 
thanks to \eqref{e:derivative} and \eqref{e:laplacian}, where we denoted $U_t= (X_t, [E_{1,t},\ldots,E_{n,t}])$. Thus we have \eqref{e:intrinsicSDE} with 
\[ M^f_t = \int_0^{t} E_{i,s}(f\circ u)\,dW^{i}_s\,,\] 
which is a martingale with quadratic variation 
\[ \langle M^f\rangle_t = \sum_{i=1}^{n}\int_0^{t} \left (E_{i,s}(f\circ u)\right )^{2}\,ds = \int_0^{t}|\nabla_g (f\circ u)(X_s)|_g^{2}\,ds\,,\]
since $\{E_{i,s}\}_{i}$ is an orthonormal basis of $T_{X_s}M$.

To show \eqref{e:extrinsicSDE} we apply the conversion formula to \eqref{e:extrinsicBM} and find in the same way 
\begin{equation}\label{e:itotostrat} df(Z_t) = P_a\vert_{Z_t}(f)dB^a_t  + \frac{1}{2}\sum_{a=1}^{q}P_{a}\vert_{Z_t}\left (P_a (f)\right )\,dt\,.\end{equation} 
Now fix a point $p\in\Sigma$ and let $\{F_i\}_{i=1}^{n}$ be a geodesic frame for $T\Sigma$ in a neighbourhood $U\subset \Sigma$ of $p$, i.e., $\{F_i\vert_q\}_{i=1}^{n}$ is an orthonormal (with respect to $h=\iota^{\sharp}e$) basis of $T_q\Sigma$ for any $q\in U$, and $\nabla^{\Sigma}_{F_i}F_j\vert_p =  0$ for every $i,j$, where $\nabla^{\Sigma}$ is the Levi-Civita connection for the metric $h$ (which is simply $\nabla^{\Sigma} = \overline\nabla^{T}$ for the Euclidean connection $\overline\nabla$). It then holds that 
\[ P_a = \sum_{i=1}^{n}\langle F_i, \frac{\partial}{\partial y^a}\rangle F_i \] 
on $U$. Thus we find 
\begin{align*} \sum_{a=1}^{q} \left (P_a(f)\right )^{2} &= \sum_{a=1}^{d}\sum_{i,j=1}^{n} \langle F_i,\frac{\partial}{\partial y^a} \rangle \langle F_j, \frac{\partial}{\partial y^a}\rangle F_i(f)F_j(f) = \sum_{i=1}^{n} F_i(f)^{2} \\
& =  |\nabla_{\iota^{*}e} f|^{2}\,.\end{align*}
On the other hand, at $p$, it holds
 \begin{align*}
\sum_{a=1}^{q}P_a\left (P_a(f)\right ) &= \sum_{a=1}^{q} \sum_{i=1}^{n} \langle F_i,\frac{\partial}{\partial y^a} \rangle F_i \left (P_a(f)\right )  \\&=\sum_{i,j}\sum_{a=1}^{q}\langle F_i,\frac{\partial}{\partial y^a} \rangle \left (\langle F_j,\frac{\partial}{\partial y^a}\rangle F_i\left (F_j(f)\right )+ F_i (\langle F_j,\frac{\partial}{\partial y^a} \rangle ) F_j(f)  \right ) \\
&= \sum_{i=1}^{n}F_i(F_i(f)) + \sum_{i,j} \langle F_i, \overline \nabla_{F_i} F_j\rangle F_j(f) \\
&= \Delta_h f + \sum_{i,j} \langle F_i, \nabla^{\Sigma}_{F_i} F_j\rangle F_j(f) = \Delta_h f\,.
\end{align*}
Since $p\in\Sigma$ is arbitrary this shows \eqref{e:extrinsicSDE} with 
\[N_t^{f} := \int_0^{t}P_a\vert_{Z_s}(f)dB^{a}_s\,.\qedhere\]  \end{proof} 

\begin{lemma}\label{l:iso}
If $u$ is isometric it holds 
\begin{equation}\label{e:gradients}
 |\nabla_g (f\circ u)(p)|_g = |\nabla_{h} f(u(p))|\,, \end{equation}
 \begin{equation}\label{e:laplaces}\Delta_g(f\circ u)(p) = \Delta_{h}f(u(p))\end{equation} for any $p\in M$ and $f\in C^\infty(\Sigma)$. 
\end{lemma}

\begin{proof}
Fix $p\in M$ and a local orthonormal frame $\{E_i\}$ for $TM$ around $p$. We denote by $u_{*}E_i\in \mathfrak{X}(\Sigma)$ the pushforward vector field (acting on $f\in C^{\infty}(\Sigma)$ by $u_{*}E_i\vert_{u(p)}(f)= E_i\vert_p (f\circ u)$). Then 
\[ |\nabla_g (f\circ u)(p)|^{2} = \sum_{i=1}^{n} \left (E_i\vert_p (f\circ u)  \right )^{2} = \sum_{i=1}^{n} \left (u_*E_i\vert_{u(p)}(f)\right )^{2} = |\nabla_{h} f(u(p))|^{2}\,,\]
since $u$ is isometric and thus $\{u_{*}E_i\}$ is a local orthonormal frame for $T\Sigma$ around  $u(p)$. Similarly we can show \eqref{e:laplaces}. Recall that for a function $f\in C^{\infty}(M)$ 
\[ \Delta_g f = \mathrm{trace}_g \nabla^{2} f = \sum_{i=1}^{n} E_i(E_i(f)) - \nabla_{E_i}E_i(f)\,,\]
for any orthonormal frame $\{E_i\}$. If, as in the proof of Lemma \ref{l:SDEs}, we choose $\{E_i\}$ to be a local geodesic frame around $p$ we find that 
\[\Delta_g(f\circ u)(p) = \sum_{i=1}^{n}E_i\vert_p(E_i(f\circ u))= \sum_{i=1}^{n} u_{*}E_i\vert_{u(p)} \left (u_*E_i(f)\right ) = \Delta_h f(u(p))\,,\]
since $\{u_{*}E_i\}$ is a local geodesic frame for $T\Sigma$ around  $u(p)$. Indeed, by the Gauss formula it holds 
\[ \nabla^{\Sigma}_{u_{*}E_i}u_{*}E_j\vert_{u(p)} = du_p\left (\nabla^{M}_{E_i} E_j\right )= 0\,\]
for every $i,j$, finishing the proof of the lemma. 
\end{proof}

At this point we define (analogous to \cite{StroockVaradhan} (where $N=\R^d$)), for a  smooth Riemannian manifold $(N,h)$ with metric $h$ of class $C^{1}$ and $y\in N$, a  \emph{solution of the martingale problem for $\frac{1}{2}\Delta_{h}$ starting at $y$} as a probability measure $P^{y}$ on $(C([0,\infty[,N),\mathcal B)$ such that 
\begin{enumerate}
\item $P^{y}(\gamma(0)=y) =1 $
\item for any $f\in C^{\infty}(N)$, the process 
\begin{equation}\label{e:martingaleproperty}
M^{f}_t = f(\eta_t) - f(y) -\frac{1}{2}\int_0^{t}\Delta_{h} f (\eta_s)\,ds 
\end{equation}
is a martingale with respect to $P^{y}$ and the filtration $\mathcal B_t$. 
\end{enumerate} 
Here, $\mathcal B$ is the Borel $\sigma$-algebra on the path space $C([0,\infty[,N)$, $\mathcal B_t$ is the canonical filtration and $\eta_t:C([0,\infty[,N) \to N$ is the continuous map $\eta_t(\gamma) = \gamma(t)$ (see Subsection \ref{s:pathspace}).

Combining Lemma \ref{l:SDEs} and Lemma \ref{l:iso} it is apparent that if the embedding $u$ is isometric, then the laws of the processes $Y_t$ and $Z_t$ are both solutions to the martingale problem for $\frac{1}{2}\Delta_h$ starting at $u(x)$. 
The following uniqueness statement for solutions of the martingale problem then yields the equality of the laws of $Y_t$ and $Z_t$ when $u$ is isometric.

\begin{lemma}\label{l:mpuniqueness} Let $(N,h)$ be a smooth, closed Riemannian manifold  with $C^{1,\alpha}$ metric $h$ for some $\alpha\in ]0,1]$. Then for any $y\in N$ the martingale problem for $\frac{1}{2}\Delta_h$ starting at $y$ has at most one solution.
\end{lemma}

\begin{proof}
Observe that for any $\varphi \in C^{\infty}(N)$ and any $\lambda >0$ it is straightforward to show the existence of a weak solution $u_\lambda \in W^{1,2}(N)$ to 
\[\lambda u -\frac{1}{2}\Delta_h u = \varphi\,\]
by energy methods (see e.g. \cite{Aubin}). Since the coefficients of $\frac{1}{2}\Delta_h$ in any coordinate chart are $C^{0,\alpha}$-functions, Schauder theory implies that in fact $u_\lambda\in C^{2,\alpha}(N)$. We can then argue as in the proof of Theorem 5.1.3 of \cite{Stroock} (which treats the case $N=\R^{n}$) to conclude. 
\end{proof} 

Conversely, assume the processes $Y_t$ and $Z_t$ have the same law for each $x \in M$. Since $u$ is an embedding, this implies that the processes $X_t$ and $\bar X_t := u^{-1}(Z_t)$ have the same law. On the other hand, from Lemmata \ref{l:SDEs} and \ref{l:iso} we find that $\bar X_t$ is Brownian motion on $(M,u^{\sharp}e)$. 

Recall that if $h$ is a smooth  Riemannian metric and $X_t^x$ is Brownian motion on $(M,h)$ starting at $x$, then the function $v(t,x) = \mathbb E[ f(X^x_t)]$ solves the heat equation on $(M,h)$ with initial datum $f\in C^\infty(M)$ (see Theorem 3.1. in  Chapter V of \cite{Ikeda}). Thanks to the uniqueness of solutions it therefore holds
\begin{equation}\label{e:BMandheat}
    \mathbb E[ f(X^x_t)] = \int_M k_h(t,x,y)f(y)d\mathrm{vol}_h(y)\,,
\end{equation}
where $k_h$ is the heat kernel of $h$ (see also Subsection \ref{s:heateq}).
For a given $f\in C^\infty(M)$ it therefore follows from \eqref{e:BMandheat} and the equality in law of the processes $X_t$ and $\bar X_t$ that 
\begin{equation}
\label{e:varadhan-integral}
\int_M k_g(t,x,y)f(y)d\text{vol}_g(y) = \int_M k_{u^{\sharp} e}(t,x,y)f(y)d\text{vol}_{u^\sharp e}(y)\,,
\end{equation} 
and hence by localization 
\[ k_g(t,x,y) \sqrt{|g|}(y) =  k_{u^\sharp e}(t,x,y) \sqrt{|u^\sharp e|}(y)\,\]
for any $t\geq 0 $ and any $x,y \in M$ close enough. 
Varadhan's lemma then implies 
\[ d^2_g(x,y)=  -2\lim_{t\to 0} t \log k_g(t,x,y) = -2\lim_{t\to 0} t \log \left(k_g(t,x,y) \sqrt{|g|}(y) \right ) = d^2_{u^\sharp e} (x,y)\,.\]
Since the Riemannian metrics $g$ and $u^\sharp e$ are smooth, 
we can recover them from the geodesic distances through differentiation separately in $x$ and $y$, followed by passage to the limit $x=y$. The reader unfamiliar with this calculation may find the details in~\cite[Lemma,p.380]{BBG}, observing that $g, u^\sharp e \in C^2$ suffices to justify the interchange of limits. 
\begin{remark}
In Lemma~\ref{l:distanceiso} below, we present an independent argument to recover the metric from geodesic distances without differentiation of the distance functions.
\end{remark}

\begin{remark}
Hsu has established a large deviation principle (LDP) in the context studied here~\cite[Thm. 2.2]{Hsu-geod}. However, we do not use this LDP because of the need to explicitly monitor the regularity assumptions on $g$ and $u$ in our work, as well as the fact that short-time asymptotics of heat kernels suffice for our needs. 
\end{remark}

\section{Proof of Theorem \ref{t:curvature}}
\label{sec:proof-curvature}
It follows from \eqref{e:extrinsicSDE} (choosing $f= y^b$)  that
\[ dZ^{b}_t= P_a(y^b)dB^{a}_t + \frac{1}{2}\Delta_\Sigma y^b dt  \,.\]
The proof of Theorem \ref{t:curvature} then follows from the identity 
\begin{equation}\label{e:meancurvatureid} \Delta_{\Sigma} y = H\end{equation}
found in the mean curvature flow literature (see e.g. \cite{Ecker}). For the reader's convenience, we prove this   identity. 

Fix $p\in \Sigma$ and a local orthonormal frame $\{F_i\}_{i=1}^{n}$ for $T\Sigma$ around $p$. By definition, it holds at $p$ 
\begin{align*}\Delta_\Sigma y^b  &= \sum_{i=1}^{n}\left ( F_i(F_i(y^b)) - \nabla^{\Sigma}_{F_i }F_i (y^b) \right )  \\
&=  \sum_{i=1}^{n}\left ( F_i(F_i(y^b)) - \overline{\nabla}_{F_i }F_i (y^b) \right ) + H(y^b)\,,\end{align*}
where $\overline{\nabla}$ denotes the covariant derivative in $\R^q$. On the other hand, writing $F_i = F_i^{c}\frac{\partial}{\partial y^c} $, we see 
\[ \overline{\nabla}_{F_i}F_i(y^b) = F_i(F_i^c)\frac{\partial}{\partial y^c}(y^b) + F_i^c\overline{\nabla}_{F_i}\frac{\partial}{\partial y^c}(y^b)= F_i(F_i^{b})\,,\]
since $\overline{\nabla}$ is flat. But $F_i^b = F_i(y^b)$, so that $\Delta_\Sigma y^b = H(y^b) = H^{b}$ at $p$ for any $b=1,\ldots,q$, which concludes the proof.

\section{Proof of Theorem \ref{t:intrinsicBMlowreg}}
\label{sec:proof3}
\subsection{Preliminaries}
\begin{subsubsection}{Path space}\label{s:pathspace}
For a given $(M,g)$ Riemannian manifold we let $C([0,\infty[,M)$ denote the space of continuous paths $\gamma:[0,\infty[\to M$, as usual equipped with the metric
\begin{equation}\label{e:metriconpathspace}\rho_g(\gamma,\tilde \gamma) = \sum_{n=1}^{\infty}\frac{1}{2^{n}} \max_{0\leq t\leq n}\min\{\mathrm{d}_g\left (\gamma(t),\tilde \gamma(t)\right ),1\}\,,\end{equation}
(where $d_g$ denotes the Riemannian distance function), under which  $C([0,\infty[,M)$ is a complete, seperable metric space with the topology of compact convergence. The Borel-$\sigma$ algebra on $C([0,\infty[, M)$ is denoted by $\mathcal B =\mathcal B(C([0,\infty[,M))$. The canonical filtration is given by $\mathcal B_t= \varphi_t^{-1}\mathcal B$ for $\varphi_t:C([0,\infty[,M)\to C([0,\infty[,M)$ defined through $\varphi_t \gamma (s) = \gamma(\min\{s,t\})$. For any $t\geq 0$ we let $\eta_t:C([0,\infty[,M) \to M$ be the continuous map $\eta_t(\gamma) = \gamma(t)$. 
\end{subsubsection}

\subsubsection{H\"older norms} Let $\Omega \subset \R^{n}$ be an open set and $f$ a real valued, vector valued, or tensor valued map defined on $\Omega$. In every case, the target is equipped with the Euclidean norm, denoted by $|f(x)|$. The H\"older norms are then defined as follows:
\begin{equation*}
\|f\|_0=\sup_{\Omega}|f|, ~~\|f\|_m=\sum_{j=0}^m\max_{|\beta|=j}\|\partial^\beta f\|_0,,
\end{equation*}
where $\beta$ denotes a multi-index,
and
\begin{equation*}
[f]_{\theta}=\sup_{x\neq y}\frac{|f(x)-f(y)|}{|x-y|^\theta},~~[f]_{m+\theta}=\max_{|\beta|=m}\sup_{x\neq y}\frac{ |\partial^\beta f(x)-\partial^\beta f(y)|}{|x-y|^\theta}, 0<\theta\leq1.
\end{equation*}
Then the H\"older norms are given as
$$\|f\|_{m+\theta}=\|f\|_m+[f]_{m+\theta}. $$

%
%

\subsubsection{Mollification estimates}
In the proof of Theorem \ref{t:renorm}
we will regularize the embedding by convolution with a standard mollifier, i.e., a radially symmetric $\varphi_\varepsilon\in C^{\infty}_c(B_\varepsilon(0))$ with $\int\varphi_\varepsilon = 1$, where $\varepsilon>0$ denotes the length-scale. Such a regularization of H\"older functions enjoys the following estimates (for a proof, see for example \cite{CDS},\cite{DIS}).
\begin{lemma}\label{l:mollification}
For any $p\geq  0,$ and $0<\alpha\leq1,$ we have
\begin{align}
&\|f-f*\varphi_\varepsilon\|_0\leq C\varepsilon^{\alpha}\|f\|_{\alpha}, \nonumber,\\
&\|(fg)*\varphi_\varepsilon-(f*\varphi_\varepsilon)(g*\varphi_\varepsilon)\|_p\leq C\varepsilon^{2\alpha-p}\|f\|_\alpha\|g\|_\alpha \label{e:commutator}
\end{align}
with constant $C$ depending only on $d, p, \beta, \varphi.$
\end{lemma}

\subsubsection{H\"older norms and mollification on manifolds}\label{s:mollificationonmanifolds}
Using a partition of unity, H\"older spaces and mollification can be defined on the compact manifold $\mathcal M$ as follows. We fix a finite atlas of $\mathcal M$ with charts $(\Omega_i,\phi_i)$, we let $\{\chi_i\}$ be a partition of unity subordinate to $\{\Omega_i\}$ and set
\[ \|f\|_k = \sum_i \|(\chi_i f) \circ \phi_i^{-1}\|_k\,,\]
and
\[ f\ast \varphi_\varepsilon = \sum_i \left ((\chi_i f)\circ \phi_i^{-1}\ast \varphi_\varepsilon\right )\circ \phi_i\,.\]
One can check that the estimates of Lemma \ref{l:mollification} still hold (with constants which may depend on the fixed charts). 

A straightforward consequence of the commutator estimate \eqref{e:commutator} is the following 
\begin{lemma} \label{l:christoffels} Let $(M,g)$ be a smooth compact manifold with Riemannian metric $g\in C^{2}$ and let $u:M\to \R^{q}$ be an isometric embedding of regularity $C^{1,\alpha}$. Denote by $u^{\varepsilon}$ the mollification of $u$ as described above. Then 
\begin{equation}\label{e:convergenceofmetrics}
\|g-(u^{\varepsilon})^{\sharp}e \|_{1}\leq C\varepsilon^{2\alpha-1} \,.
\end{equation}
\end{lemma}

The short proof of the previous lemma is contained in \cite{CDS} Proposition 1.

\subsection{Proof of Theorem \ref{t:intrinsicBMlowreg}}
\begin{lemma} \label{l:tightness} For any $x\in M$, the sequence $\{P^{k,x}\}_{k\in \N} $ is tight and therefore relatively compact in the set of probability measures on $C([0,\infty[,M)$ with respect to weak convergence.
\end{lemma} 
We prove Lemma~\ref{l:tightness} in the next subsection. We prove Theorem \ref{t:intrinsicBMlowreg} using Lemma~\ref{l:tightness} and the following

\begin{lemma}\label{l:convergenceoflaw} Let $\{P^{k_j,x}\}_{j\in \N}$ be a subsequence which converges weakly to a probability measure $P^{x}$. Then $P^{x}$ is a solution to the martingale problem for $\frac{1}{2}\Delta_g$ starting at $x$.
\end{lemma}

Combined with the uniqueness assertion of Lemma \ref{l:mpuniqueness} this Lemma shows that the full sequence $\{P^{k,x}\}_{k\in \N}$ converges weakly to the solution of the martingale problem for $\frac{1}{2}\Delta_g$ starting at $x$, finishing the proof of Theorem \ref{t:intrinsicBMlowreg}. 
%


\begin{proof}[Proof of Lemma \ref{l:convergenceoflaw}] The argument is similar the one used to prove the existence of weak solutions of stochastic differential equations with bounded, continuous coefficients (see e.g. Theorem 5.4.22 in \cite{Karatzas}). For the reader's convenience we present the argument. Firstly, let us abbreviate $P^{j}:= P^{k_j,x}$ and observe that since $P^{j}$ is the law of Brownian motion $X^{k_j,x}$ starting at $x$ it holds $P^{j}(\gamma(0)=x) =1$. By the weak convergence we therefore find for any $f\in C^{\infty}(M)$ that 
\[\int_{C([0,\infty[,M)}f(\gamma(0))\,dP^{x}= \lim_{j\to\infty} \int_{C([0,\infty[,M)}f(\gamma(0))\,dP^{j} = f(x)\,,\]
from which $P^{x}(\gamma(0)=x) =1 $ follows. It remains to show \eqref{e:martingaleproperty}, or, equivalently,  that for any $0\leq s<t<\infty $ and any $\mathcal B_s$ measurable $F\in C_b(C([0,\infty[,M))$ it holds  
 \[ \mathbb E^{P}\left [ \left (M^{f}_t-M^{f}_s\right )F \right ] =0\,,\]
 where $M^{f}_t $ is as in \eqref{e:martingaleproperty} with $h=g$. 
 Fix then $0\leq s<t<\infty $ and $F\in C_b(C([0,\infty[,M))$ which is $\mathcal B_s$ measurable. Observe that since $X^{k}$ is Brownian motion on $(M,g_k)$, it holds 
 \begin{equation}\label{e:martingalej}\mathbb E^{P^{k_j}}\left [ \left (M^{f,j}_t-M^{f,j}_s\right )F \right ] =0\,,\end{equation} where $M^{f,j}_t:C([0,\infty[,M)\to \R$ is the process defined by
\[ M^{f,j}_t(\gamma) = f(\gamma(t))-f(x)-\frac{1}{2}\int_0^{t}\Delta_{g_{k_j}}f(\gamma(s))\,ds\,.\]
We will now show that $G_j:C([0,\infty[,M)\to\R$ given by 
\[ G_j(\gamma) = M^{f,j}_t(\gamma)-M^{f,j}_s(\gamma)\,\] 	
is a uniformly bounded sequence of continuous functions converging uniformly on $C([0,\infty[,M)$ to $G(\gamma) = M^{f}_t(\gamma)-M^{f}_s(\gamma)$. We can then pass to the limit in \eqref{e:martingalej} to conclude (see e.g. Problem 2.4.12 in \cite{Karatzas}). Indeed, observe first that if $d_g(\rho(t),\tilde \rho(t))$ is smaller than the injectivity radius of $M$ (which is strictly positive since $M$ is compact without boundary),  we get
\[ |M^{f}_t(\gamma)-M^{f}_t(\tilde \gamma)| \leq \left (\|f\|_1+ Ct\|f\|_3\right )d_g(\gamma(t),\tilde\gamma(t))\leq Ct\|f\|_3 \rho_g(\gamma,\tilde \gamma) \]
for some constant $C$  depending on $M$ and $g$. A similar estimate holds for $M^{f,j}_t$, thus $G_j$ is uniformly continuous. For the convergence we  observe that for any $k\geq 1$ large enough and any function $h\in C^{2}(M)$ we can estimate
\begin{equation}\label{e:laplacianestimate} \|(\Delta_g-\Delta_{g_k})h\|_0 \leq C\|g-g_k\|_{1}\|h\|_2\,,\end{equation}
for some constant depending on $M$ and $g$, 
as can be seen from the coordinate expression of the Laplace-Beltrami operator. This implies
\[ \sup _{\gamma \in C([0,\infty[,M)}|G_j(\gamma)-G(\gamma)|\leq C|t-s|\|g-g_k\|_1\|f\|_2 \to 0\,\]
 for $k\to \infty$, finishing the proof.

\end{proof}

\subsection{Proof of Lemma \ref{l:tightness}}
We use the following characterization of tightness following \cite{Karatzas}. Let us introduce the following modulus of continuity for fixed $T>0,\delta>0, \gamma\in C([0,\infty[,M)$:
\begin{equation}\label{d:modofcont} w^{T}_\delta(\gamma) = \sup_{|s-t|<\delta, 0\leq s,t\leq T} d_g(\gamma(t),\gamma(s))\,.\end{equation}
We then have
 \begin{lemma}
 A sequence $\{P^{k}\}_{k\in \N}$ of probability measures on $C([0,\infty[,M)$ is tight if for every $\varepsilon>0$ small enough and every $T>0$ it holds 
 \begin{equation}\label{e:Aldous}
 \lim_{\delta\downarrow 0} \sup_{k\geq 1}P^{k}\left (w^{T}_\delta\geq \varepsilon\right ) =0 \,.
 \end{equation}
 \end{lemma}
The proof of the preceding lemma is a simple modification of the argument for the case $M=\R$ given in Theorems 2.4.9 and 2.4.10 in \cite{Karatzas}. We now prove Lemma \ref{l:tightness} by showing that \eqref{e:Aldous} holds for the sequence $\{P^{k,x}\}$ of laws of the intrinsic Brownian motions $X^{k,x}$. Since the metric $g_k$ is smooth and the corresponding SDE \eqref{eq:intrinsic} has unique strong solutions given arbirtary initial values, we can assume that the processes $X^{k,x}$ are defined on a common probability space $(\Omega,\mathcal F,\mathbb P)$ and are adapted with respect to a common filtration $\mathcal F_t$. Unwinding the definitions, it is clear that to show \eqref{e:Aldous} we need an estimate on 
\[ \mathbb P(d_g(X^{k,x}_t,X^{k,x}_s)\geq \varepsilon ) \] 
for arbitrary $|s-t|<\delta$ small. To estimate the latter quantity we follow \cite{Ma} (Section 4.3) and introduce, for arbitrary $x\in M, k\in \N, \delta>0$, the exit times 
\begin{equation}\label{d:exit times}
\tau_{\varepsilon}^{k,x} = \inf\{t>0: d_g(X^{k,x}_t,x)>\varepsilon\}\,.
\end{equation}
We claim that
\begin{lemma} \label{l:exittimes} There exists $\varepsilon_M>0$ only depending on $(M,g)$ such that for any $0<\varepsilon<\varepsilon_M$ there exists a constant $C_{\varepsilon}>0$ (depending only on $\varepsilon,M$ and $g$) such that  
\begin{equation}\label{e:exittimeestimate}
\sup_{x\in M}\mathbb{P}( \tau^{k,x}_{\varepsilon}\leq t) \leq C_{\varepsilon} t 
\end{equation}
for any $t\geq 0$ and $k\in \N$ large enough.
\end{lemma}

We postpone the proof of the preceding lemma and show how to infer \eqref{e:Aldous}. Observe that, since $X^{k,x}$ is Brownian motion on $M$, the coordinate process $\eta =\{\eta_t\}_{t\geq 0}$ on $(C([0,\infty[,M), \mathcal B)$ together with the family of laws $\{P^{k,x}\}_{x\in M}$ is a Markov family. Fixing $T>0$, $\delta>0$, $0\leq s\leq t\leq T$ with $t-s<\delta$, and $\varepsilon<\varepsilon_M$, we can then estimate, using the Markov property 
\begin{align*}
\mathbb P \left ( d_g(X^{k,x}_t,X^{k,x}_s)\geq \varepsilon\right )&= P^{k,x}\left ( d_g(\eta_t,\eta_s)\geq \varepsilon \right ) \\
&= \int_{M}P^{k,x}\left (d_g(\eta_t,\eta_s)\geq \varepsilon \vert \eta_s = y\right )\,d(\eta_s)_{*}P^{k,x}(y) \\
&= \int_{M}P^{k,y}\left ( d_g(\eta_{t-s},y)\geq \varepsilon\right )\,d(\eta_s)_{*}P^{k,x}(y) \\
\end{align*}
But due to \eqref{e:exittimeestimate} we have 
\begin{align*} \sup_{y\in M}P^{k,y}\left ( d_g(\eta_{t-s},y)\geq \varepsilon\right ) &= \sup_{y\in M}\mathbb{P} \left (d_g( X^{k,y}_{t-s},y)\geq \varepsilon\right ) \leq \sup_{y\in M}\mathbb{P} \left (\tau^{k,y}_{\varepsilon}\leq t-s\right )  \\
&\leq C_\varepsilon (t-s)\\
&\leq C_\varepsilon \delta\,,
\end{align*}
which implies \eqref{e:Aldous} for the sequence $\{P^{k,x}\}_{k\in \N}$ and hence also the statement of  Lemma \ref{l:tightness}. We are therefore left to show Lemma \ref{l:exittimes}. 
\begin{proof}[Proof of Lemma \ref{l:exittimes}]
We want to use the property \eqref{e:martingaleprop} for $X^{k,x}$ to gain information about the exit-time of the process $X^{k,x}$. In order to do so we need to construct a suitable testfunction $f$. Since the manifold $(M,g)$ is closed, the injectivity radius is bounded from below by a constant $\varepsilon_M>0$. In particular, for any fixed $x\in M$, the Riemannian distance function $d_g(x,\cdot)$ is $C^{2}$ on the open set $B_\varepsilon(x)\setminus \{x\}$, where $B_\varepsilon(x)$ is the open geodesic ball with radius $\varepsilon$ (this follows for example from the formula (4.2) in \cite{Villani}).  Given $0<\varepsilon<\varepsilon_M$ we let $\theta\in C^{\infty}_c(\R)$ be a cutoff function with $0\leq \theta \leq 1$ on $\R$, $\theta_{\varepsilon}\equiv 1 $ on $]-\frac{\varepsilon}{4},\frac{\varepsilon}{4}[$ and $\theta_{\varepsilon}\equiv 0$ outside $]-\frac{\varepsilon}{2},\frac{\varepsilon}{2}[$. Then the function \[ f^{x}_{\varepsilon} := \theta_{\varepsilon}\circ d_g(x,\cdot) \,\] 
is a $C^{2}$-function on $M$ with support contained in $B_\varepsilon(x)$. As in \eqref{e:laplacianestimate}, we can estimate for $k$ large enough 
\[ \|\Delta_{g_k} f^{x}_{\varepsilon} \|_0 \leq C \|f^{x}_{\varepsilon}\|_2 \leq C_\varepsilon\] 
for a constant $C_\varepsilon$ only depending on $\varepsilon$ and $(M,g)$, but not on $k$ or $x$. Now fix $t\geq 0$. We use that $X^{k,x}$ is the intrinsic Browinan motion on $(M,g_k)$ and the fact that $0\leq f^{x}_\varepsilon \leq 1 $ with $f^{x}_\varepsilon(x) =1$ and $f^{x}_\varepsilon =0$ outside $B_{\frac{\varepsilon}{2}}(x)$ to find
\begin{align*}
\mathbb P( \tau ^{k,x}_{\varepsilon} \leq t) &= 1-\mathbb P( \tau ^{k,x}_{\varepsilon} > t) \\
&\leq 1- \mathbb E\left (\chi_{\{\tau ^{k,x}_{\varepsilon}> t\}}f^{x}_\varepsilon\left(X^{k,x}_{\tau^{k,x}_{\varepsilon}\wedge t} \right )\right )\\
&=1 - \mathbb E\left (f^{x}_{\varepsilon}\left (X^{k,x}_{\tau^{k,x}_{\varepsilon}\wedge t} \right )\right ) +\mathbb E\left (\chi_{\{\tau ^{k,x}_{\varepsilon}\leq t\}}f^{x}_{\varepsilon}\left (X^{k,x}_{\tau^{k,x}_{\varepsilon}\wedge t} \right )\right )\\
&= f^{x}_{\varepsilon}(x)- \mathbb E\left (f^{x}_{\varepsilon}\left (X^{k,x}_{\tau^{k,x}_{\varepsilon}\wedge t} \right )\right )\\
&= -\mathbb E \int_{0}^{\tau^{k,x}_{\varepsilon}\wedge t}\frac{1}{2}\Delta_{g^{k}} f^{x}_{\varepsilon}(X^{k,x}_s)\,ds\leq C_\varepsilon t
\end{align*}
for some constant $C_\varepsilon$ depending only on $M,g$ and $\varepsilon$. This concludes the proof.
\end{proof}
\begin{remark}
\label{r:smooth-g} 
The assumption that $g\in C^2$ is used in the proof of Lemma~\ref{l:exittimes}.
\end{remark}

\subsection{Proof of Corollary \ref{c:markov1}}
Consider the family  $\{P^x\}_{x\in M}$, where $P^x$ is the unique solution to the martingale problem for $\frac{1}{2}\Delta_g$ starting at $x\in M$. By Theorem 5.1 in \cite{Ikeda} it is (strongly) Markovian. In particular, for any $t>s\geq 0$, any set $C\in \mathcal{B}_t$  and any $\Gamma \in\mathcal B(M)$ it holds 
\[ P^x( C\cap \{\gamma:\gamma(t)\in \Gamma\})  = \int_C P^{\gamma'(s)}(\{\gamma: \gamma(t-s)\in \Gamma\}) \,dP^x(\gamma')\,.\]
To show that $\{X^x\}_{t\geq 0}$ is Markovian with transition density given by the heat kernel we need to check that 
\[ \mathbb E [ \chi_\Gamma(X^x_{t+s}) \vert \mathcal F_t] = \int_\Gamma k_g(s, X_t^x,y)d\mathrm{vol}_g(y)\,,\]
for fixed $s >t\geq 0$, $\Gamma\in \mathcal B(M)$
where $\chi_\Gamma$ is the indicator function of $\Gamma$. Observe first that the right hand side equals $v(X^x_t,s)$, where $v(x,t) =\int_\Gamma k_g(t,x,y)\,d\mathrm{vol}_g(y)$ is the solution of the heat equation on $(M,g)$ with initial datum $\chi_\Gamma$. Since $g\in C^2$ it follows by Schauder theory that $v\in C^2$. In particular, $v(X^x_t,s)$ is $\mathcal F_t $ measurable for any $s$. Now observe that $\mathcal F_t$ is generated by sets $(X^x)^{-1}(C)$, where $C\in \mathcal B_t$ is a cylinder set. For such a set $C$ it holds 
\[ \int_{(X^x)^{-1}(C)} \chi_\Gamma(X^x_{t+s})\,d\mathbb P = \int_C \chi_\Gamma\circ \eta_{t+s}  \,dP^x = P^x(C\cap \{\gamma:\gamma(t+s)\in \Gamma\})\,.\]
By the Markov property of $P^x$ it holds 
\begin{align*}
P^x(C\cap \{\gamma:\gamma(t+s)\in \Gamma\}) &= \int_C P^{\gamma'(t)}(\{ \gamma: \gamma(s)\in \Gamma\}) \,dP^x(\gamma') \\ 
&= \int_{(X^x)^{-1}(C)} P^{X^x_t} (\{\gamma:\gamma(s)\in \Gamma\})\,d\mathbb{P}\,.
\end{align*}
On the other hand, from $v\in C^2$ and the martingale property \eqref{e:martingaleprop} it follows as usual that $v(x,s) = \mathbb E [\chi_\Gamma(X^x_s)] =P^x(\{\gamma:\gamma(s)\in \Gamma\})$ for any $x\in M$ and $s\geq 0$. Combining with the above yields 
\begin{align*}
    \int_{(X^x)^{-1}(C)} \chi_\Gamma(X^x_{t+s})\,d\mathbb P &=  \int_{(X^x)^{-1}(C)} P^{X^x_t} (\{\gamma:\gamma(s)\in \Gamma\})\,d\mathbb{P} \\
    &= \int_{(X^x)^{-1}(C)} v(X^x_t,s)\,d\mathbb P\\
    & = \int_{(X^x)^{-1}(C)} \left(\int_\Gamma k_g(s, X_t^x,y)d\mathrm{vol}_g(y)\right) \,d\mathbb{P} \,,
\end{align*}
finishing the proof.

\section{Proof of Theorem \ref{t:renorm}}
Theorem~\ref{t:renorm} is similar in spirit to Theorem~\ref{t:main}. However, we must now use PDE theory to establish the existence of the heat kernels since the pullback metric $u^{\#}e$ is only $C^{0,\alpha}$ for a $C^{1,\alpha}$ embedding. In the proof below, Lemma~\ref{l:heatconvergence}  is used to establish the analogue of equation~\eqref{e:varadhan-integral}. We then use Norris' version of Varadhan's lemma~\cite{Norris} to recover geodesic distance from the short-time asymptotics for heat kernels. But an additional step, Lemma~\ref{l:distanceiso}, is needed to conclude equality of the metrics from equality of geodesic distances. 

This section concludes in Lemma~\ref{l:markov-bm} on the Markov property for Brownian motion on Riemannian manifolds with $C^0$ metric. Its proof uses the change of variables formula for heat kernels given in the preliminary Lemma~\ref{l:heatkernels}.  Lemma~\ref{l:markov-bm} is then used to establish Corollary~\ref{cor:equivalence}.

\subsection{Preliminaries}
\subsubsection{Heat equation on Riemannian manifolds}\label{s:heateq}
Let $h$ be a continuous metric on $M$. Then the Laplace-Beltrami operator $-\Delta_h$ is well defined as an unbounded operator on $L^2(M)$. Recall that for a $v\in W^{1,2}(M)$, the  action on $-\Delta_h v$ on a test function $\varphi\in W^{1,2}(M)$ is given by 
\[ \langle-\Delta_h v, \varphi\rangle = \int_M h(\nabla_h v, \nabla_h \varphi)d\text{vol}_h\,.\]
Classical theory (see e.g. Theorem 4.1 Chapter III in \cite{Lions} and \cite{Sturm}) asserts that for any $f\in L^2(M)$, the Cauchy problem for the heat equation $\partial_t v -\Delta_h v = 0$, $v(0)=f$ has a unique (weak) solution $v\in L^2(]0,\infty[, W^{1,2})$ and that it can be written as 
\[ v(t,x) = \int_M k_h(t,x,y)f(y)d\text{vol}_h(y)\,,\]
where $k_h$ is called the \emph{heat kernel} of $h$. 
It is a smooth function of $(t,x,y)$ when $h$ is smooth. In the proof of Theorem \ref{t:renorm} below we will consider $h= u^\sharp e$, which is only $C^{0,\alpha}$. In particular, the classical formula \eqref{e:BMandheat} does not hold a priori. We therefore need the following lemma.

\begin{lemma}\label{l:heatconvergence} Let $\{h_\varepsilon\}$ be a family of smooth Riemannian metrics on $M$ and $\{k_{h_\varepsilon}\}$ the corresponding heat kernels. Assume that $\|h_\varepsilon-h\|_0 \to 0 $ for $\varepsilon\to 0$. Then for any $f\in L^2(M)$ and any $t>0$, $x\in M$ it holds
\begin{equation}\label{e:heatconvergence}
\lim_{\varepsilon\to 0} \int_M  k_{h_\varepsilon}(t,x,y)f(y)d\mathrm{vol}_{h_\varepsilon}(y) = \int_M k_h(t,x,y)f(y)d\mathrm{vol}_h(y)\,.
\end{equation}
\end{lemma}

\begin{proof} The function $v_\varepsilon (t,x) =  \int_M  k_{h_\varepsilon}(t,x,y)f(y)d\mathrm{vol}_{h_\varepsilon}(y)$ is the unique smooth solution of the heat equation $\partial_t v_\varepsilon - \Delta_{h_\varepsilon} v_\varepsilon = 0$ on $M$ with $v_\varepsilon(0,\cdot) = f$. Testing the equation with $v_\varepsilon$ and integrating yields a uniform $L^\infty((0,T), L^2(M))\cap L^2((0,T), W^{1,2}(M)) $ bound for $\{v_\varepsilon\}$ and therefore a weakly in $L^2((0,T), W^{1,2}(M))$ converging subsequence. Passing to the limit in the weak formulation of the heat equation, we find (using the uniqueness of weak solutions) that the limit is $v(t,\cdot)$. On the other hand, the coordinate expression of $v_\varepsilon$ in some fixed chart is a smooth solution to 
\[ \partial_t \left( \sqrt{|h_\varepsilon|} v_\varepsilon\right) - \partial_i\left( \sqrt{|h_\varepsilon|} h_\varepsilon^{ij} \partial_j v_\varepsilon\right)  =0 \]
in some ball $B_r(0)\subset\R^n$ (recall that $|h| = \det h$). Due to the convergence assumption and the compactness of $M$ there exist $0<\lambda<\Lambda <+\infty$ such that for all $\varepsilon>0$ small enough it holds
 \[\lambda  \leq \sqrt{|h_\varepsilon|}\leq \Lambda\]
 and \[ \lambda Id \leq  \sqrt{|h_\varepsilon|} h_\varepsilon ^{ij}\ \leq \Lambda Id\,.\] 
The De Giorgi--Nash--Moser theorem (see Theorem 18 in \cite{Vasseur} for our setting) then implies a uniform (in $\varepsilon$) $C^{0,\alpha}([s,T]\times \bar B_{\frac{r}{2}}) $ - bound for the functions $v_\varepsilon$ for all $s>0$. By the theorem of Arzel\`a-Ascoli, a subsequence of $v_\varepsilon$ therefore converges uniformly in $(t,x)$. By the above weak convergence, the limit must be $v$, which shows the claim. 
\end{proof}

Now  let $u:M\to \R^q$ be a $C^1$-regular embedding and consider the submanifold $\Sigma = u(M)$. Consider the Riemannian metrics $u^\sharp e$ on $M$ and $\iota^\sharp e$ on $\Sigma$, where $\iota : \Sigma \hookrightarrow \R^q$ is the inclusion. Both metrics are continous and therefore induce heat kernels $k_{u^\sharp e}, k_{\iota^\sharp e}$. The relation between the two is given in the next lemma. 
\begin{lemma}\label{l:heatkernels} For any $f\in L^2(M)$ and any $t>0$ it holds
\[\int_M k_{u^\sharp e}(t,x,y)f(y) d\mathrm{vol}_{u^\sharp e}(y)= \int_\Sigma k_{\iota^\sharp e} (t,u(x),z) f(u^{-1}(z))\,d\mathrm{vol}_{\iota^\sharp e}(z)\,.\]
\end{lemma}
\begin{proof} This follows from a change of variables. Indeed, fix $f\in L^2(M)$ and define $v(x,t) =\int_\Sigma k_{\iota^\sharp e} (t,u(x),z) f(u^{-1}(z))\,d\mathrm{vol}_{\iota^\sharp e}(z)$. We claim that $v$ is a weak solution to the heat equation on $(M,u^\sharp e) $ with initial datum $f$. Indeed, fix $\varphi \in C^\infty(M)$ and observe by that the change of variables 
 \begin{align*}
     \int_M \partial_tv\, \varphi \,d\mathrm{vol}_{u^\sharp e}+ &\int_M u^\sharp e(\nabla_{u^\sharp e} v,\nabla_{u^\sharp e} \varphi) d\mathrm{vol}_{u^\sharp e} =\\
     &\int_\Sigma \partial_t v\circ u^{-1} \varphi\circ u^{-1}\,d\mathrm{vol}_{\iota^\sharp e} + \int_\Sigma u^\sharp e\vert_{u^{-1}}(\nabla_{u^\sharp e} v,\nabla_{u^\sharp e} \varphi) \mathrm{vol}_{\iota\sharp e}\,.
 \end{align*}
Now fix any $p\in M$ and local coordinates $(x^1,\ldots,x^n)$ around $p$. This induces local coordinates $y^i:= x^i\circ u^{-1}$ on $\Sigma $ around $u(p)$. Notice that  $\frac{\partial}{\partial y^i} = u_\sharp \frac{\partial}{\partial x^i}$. Consequently, 
\[ \iota^\sharp e_{u(p)} = u^\sharp e_p(\frac{\partial}{\partial x^i},\frac{\partial}{\partial x^j}) dy^idy^j \,.\]
Therefore
 \begin{align*} u^\sharp e\vert_p (\nabla_{u^\sharp e} v,\nabla_{u^\sharp e} \varphi) &= (u^\sharp e)^{ij}(p) \frac{\partial}{\partial x^i}\big\vert_p (v) \frac{\partial}{\partial x^j}\big\vert_p (\varphi) \\
 &= (\iota^\sharp e)^{ij}(u(p)) \frac{\partial}{\partial y^i}\big\vert_{u(p)} (v\circ u^{-1}) \frac{\partial}{\partial y^j}\big\vert_{u(p)} (\varphi\circ u^{-1})\, \\
 &= \iota^\sharp e\vert_{u(p)}(\nabla_{\iota^\sharp e} (v\circ u^{-1}), \nabla_{\iota^\sharp e} (\varphi\circ u^{-1}))\,.
 \end{align*}

 But $\tilde v(y,t):= v(u^{-1}(y),t)$ is (by definition) the unique weak solution of the heat equation on $(\Sigma, \iota^\sharp e)$ with initial datum $f\circ u^{-1}\in L^2(\Sigma)$. Since $\tilde \varphi := \varphi\circ u^{-1}$ is a valid testfunction it follows from the above
 \begin{align*}
     \int_M \partial_tv\, \varphi \,d\mathrm{vol}_{u^\sharp e}+ &\int_M u^\sharp e(\nabla_{u^\sharp e} v,\nabla_{u^\sharp e} \varphi) d\mathrm{vol}_{u^\sharp e} =\\
     &\int_\Sigma \partial_t \tilde v\tilde \varphi \,d\mathrm{vol}_{\iota^\sharp e} +\int_\Sigma \iota^\sharp e(\nabla_{\iota^\sharp e} \tilde v,\nabla_{\iota^\sharp e} \tilde \varphi) \mathrm{vol}_{\iota\sharp e}\, = 0\,.
 \end{align*}
 By another change of variables argument it follows that $\lim_{t\to 0} v(\cdot,t ) =f$ in $L^2(M)$, which finishes the proof due to uniqueness of solutions. \end{proof}
\subsubsection{Recovering Riemannian metrics from their distance function}
In the proof of Theorem \ref{t:main} we used that one can recover the Riemannian metric from its distance function by differentiation. This proof relies on the regularity assumption that the metric is twice differentiable. In the proof of Theorem \ref{t:renorm} we only know a priori that the metric $u^\sharp e$ is (H\"older-)continuous. The following Lemma is necessary to close this gap. 

\begin{lemma} 
\label{l:distanceiso}
Let $g\in C^2$ and $h\in C^0$ be two Riemannian metrics on $M$.
Assume that for any $x\in M $ it holds 
\begin{equation}\label{e:equidist}d_g(x,y)= d_h(x,y)\end{equation}
for all $y$ in a $g$-geodesic neighborhood of $x$. Then $g=h$. 
\end{lemma}
As usual, the distance function $d_h(x,y)$ is defined by 
\[ d_h(x,y) = \inf_\gamma \int_0^1|\gamma'|_h\,dt =:\inf_\gamma L_h(\gamma)\,,\] where the infimum is taken over all piecewise smooth curves $\gamma:[0,1]\to \M$ with $\gamma(0)=x,\gamma(1)=y$. 
\begin{proof}
Let $x\in M$, $y$ in a $g$-geodesic neighborhood of $x$ and let $\gamma:[0,L]\to M$ be the unique $g$-geodesic connecting $x$ and $y$. Using the theorem of Arzel\`a-Ascoli we can find a piecewise smooth curve $\bar \gamma:[0,1]\to M $ connecting $x $ and $y$ which minimizes the length $L_h$ (see also Section 2.1 in \cite{Saemann} and references therein). We claim that $\bar \gamma $ is a reparametrization of $\gamma$. Fix $s\in (0,1)$ and observe that  
\begin{align*}
    d_g(x,y) &\leq d_g(x,\bar\gamma (s))+d_g(y,\bar \gamma(s))  = d_h(x,\bar\gamma (s))+d_h(y,\bar \gamma(s))\\ 
    &\leq \int_0^s |\bar \gamma'|_h\,dt  +\int_s^1 |\bar \gamma'|_h\,dt  = L_h(\bar \gamma) =d_h(x,y)=d_g(x,y)\,,
\end{align*} 
i.e., $ d_g(x,y)= d_g(x,\bar \gamma(s)) + d_g(y,\bar \gamma(s))$. But this implies that $\bar \gamma (s) = \gamma(t(s)) $ for some $t(s)\in [0,L]$.  Indeed, connecting  $x$ and $\bar\gamma(s)$ by the unique $g$-geodesic and following the unique $g$-geodesic from $\bar \gamma(s)$ to $y$ yields a $d_g$-minimizing curve connecting $x,y$, implying that the curve is $\gamma$ by uniqueness of $d_g$-minimizing curves. 

Now fix $x\in M$ and $X\in T_xM$ and let $\gamma:(-1,1)\to M$ be a $g$-geodesic with $\gamma(0)=x$ and $\gamma'(0)=X$. Fix $T<1 $ such that $\gamma(T) $ belongs to a geodesic neighborhood of $x$. Let then $\bar \gamma:[0,1]\to M$ be a $L_h$-minimizing curve connecting $x$ and $\gamma(T)$, and fix $t\in (0,T)$. By the above reasoning it holds $\gamma(t)= \bar\gamma(s(t))$ for some $s(t)\in (0,1)$ and hence, using \eqref{e:equidist} and the indepence of the length of a curve with respect to reparametrizations, 
\begin{align}
    \int_0^s |\gamma'|_g\,dt &= d_g(x,\gamma(s)) = d_h(x,\bar\gamma(t(s))) = L_h(\bar\gamma\vert_{[0,t(s)]}) = L_h(\gamma\vert_{[0,s]}) \\
    &= \int_0^s |\gamma'|_h\,dt\,.
\end{align} 
Dividing by $s$ and letting $s$ approach zero then yields $|X|_g = |X|_h$, from which the claim follows. 
 \end{proof}

\subsubsection{Mollifying embeddings}
For $\varepsilon>0$ let $u^\varepsilon: M \to \R^q$ denote the mollified map $u^\varepsilon= u* \varphi_\varepsilon$ as defined in Subsection \ref{s:mollificationonmanifolds}.
\begin{lemma}
\label{lem:smoothing}
There exists $\varepsilon_*>0$ such that $u^\varepsilon$ is an embedding for $0 < \varepsilon < \varepsilon_*$.   
\end{lemma}

\begin{proof} Fix a finite atlas for $M$. Since $u$ is an immersion and $M$ is compact, there exists $\eta>0$ such that $|Du_p\xi|\geq5\eta $ for all $p\in M$ and $\xi \in \mathbb S^{n-1}$. We can estimate
\begin{equation}\label{e:immersion}|Du^\varepsilon_p\xi| \geq 5\eta -\|u-u^\varepsilon\|_1\geq 4\eta\,,\end{equation}
which shows that $u^\varepsilon$ is an immersion for $\varepsilon$ small enough, since $\|u-u^\varepsilon\|_1\leq C\varepsilon^\alpha\|u\|_{1,\alpha}$ by Lemma \ref{l:mollification}. It remains to show that $u^\varepsilon$ is injective for $\varepsilon$ small enough. First of all, since $u\in C^1$, there exists $\rho>0$ such that $|Du_p- Du_q| \leq \eta$ for all $d_g(p,q)<\rho$. Without loss of generality we assume $\rho< \text{inj}(M)$. Since $u$ is an embedding, there is $\delta>0$ such that $|u(p)-u(q)|\geq 3\delta $ for $d_g(p,q)\geq \rho$. For such $p,q$ we therefore find 
\[|u^\varepsilon(p)-u^\varepsilon(q)|\geq 3\delta- 2\|u-u^\varepsilon\|_0 \geq \delta\]
for $\varepsilon$ small enough. On the other hand, fix $d_g(p,q)<\rho$ and let $x,y$ denote the coordinates of $p,q$. By the mean value theorem we find 
\begin{align*}|u^\varepsilon(x)-u^\varepsilon(y) - Du^\varepsilon(x)(x-y)|&\leq |\left(Du^\varepsilon(x+\tau(y-x))-Du^\varepsilon(x)\right)||x-y| \\
&\leq  3\eta|x-y|
\end{align*}
for $\varepsilon$ small enough. Therefore, using \eqref{e:immersion},
\[|u^\varepsilon(x)-u^\varepsilon(y)| \geq |Du^\varepsilon(x)(x-y)|-3\eta|x-y|\geq \eta|x-y|\]
for $\varepsilon$ small enough, which yields the claim. 
\end{proof}
\subsection{Proof of Theorem \ref{t:renorm}} Let $u$ be isometric.
As in the proof of Theorem \ref{t:intrinsicBMlowreg} we can assume that the processes $X^{\varepsilon}_t$ are defined on a common probability space $(\Omega,\mathcal F, \mathbb{P})$, and hence also  $Y^{\varepsilon}_t=u^{\varepsilon}(X^{\varepsilon}_t)$. The law of $Y^{\varepsilon}$ is a probability measure on the path space $C([0,\infty[, \R^{q})$ (which is equipped with the metric $\rho_e$ defined as in \eqref{e:metriconpathspace}). Fix therefore $f:C([0,\infty[, \R^{q}) \to \R$ continuous and bounded. Then 
 \[ \int_{C([0,\infty[, \R^{q})} f(\gamma) \,d Y^{\varepsilon}_{*} \mathbb P(\gamma) = \int_{C([0,\infty[, M)} f\circ u^{\varepsilon}(\gamma) \,dX^{\varepsilon}_{*}\mathbb P(\gamma)\,.\]
Because $u$ is isometric and $\alpha>\frac{1}{2}$ it follows from Lemma \ref{l:christoffels} and Theorem \ref{t:intrinsicBMlowreg} that $X^{\varepsilon}_t$ converges in law to $X_t$ for $\varepsilon \to 0$. Hence, if we show that  $f\circ u^{\varepsilon}$ converges uniformly on compact subsets of $C([0,\infty[, M)$ to $f\circ u$, it will follow (see e.g. Problem 2.4.12 in \cite{Karatzas}) 
\begin{align*}\lim_{\varepsilon\downarrow 0} \int_{C([0,\infty[, \R^{q})} f(\gamma) \,d Y^{\varepsilon}_{*} \mathbb P(\gamma) &= \int_{C([0,\infty[, M)} f\circ u(\gamma) \,dX_{*}\mathbb P(\gamma) \\
&=  \int_{C([0,\infty[, \R^{q})} f(\gamma) \,d Y_{*} {\mathbb P}(\gamma)\,,\end{align*}
as required. Thus, fix any compact subset $K\subset C([0,\infty[,M)$ and a sequence $\varepsilon_i\downarrow 0$. Observe that $u(K)$ is compact in $C([0,\infty[,\R^q)$ (equipped with the usual topology of compact convergence) thanks to the estimate
\[ |u(\gamma_i(t))-u(\gamma(t))|\leq C[u]_1d_g(\gamma_i(t),\gamma(t))\]
if $d_g(\gamma_i(t),\gamma(t))$ is smaller than the injectivity radius of $(M,g)$. The same holds for the set $u^\varepsilon(K)$ for any $\varepsilon>0$. Since moreover 
\[ |u^\varepsilon(\gamma_i(t))-u(\gamma(t))|\leq \|u^\varepsilon-u\|_0 + C[u]_1d_g(\gamma_i(t),\gamma(t))\]
it follows that also the set $\tilde K=\bigcup_{i=1}^\infty u^{\varepsilon_i}(K)$ is compact. As a continuous function, $f$ is uniformly continuous on $\tilde K$ and given $\delta>0$ we find $\eta>0$ such that $|f(z)-f(w)|<\delta$ for any $z,w\in \tilde K$ with $ \rho_e (z,w)<\eta$. Since $\|u^\varepsilon-u\|_0 \leq C\varepsilon [u]_1$ by Lemma \ref{l:mollification}, we deduce that for any $T>0$
\[ \sup_{\gamma\in K}\sup_{0\leq t\leq T}|u_{\varepsilon_i}(\gamma(t))-u(\gamma(t))| \leq  C\varepsilon [u]_1<\eta\]
for any $i$ large enough, i.e., for such $i$ it holds $\tilde \rho (u_{\varepsilon_i}\circ\gamma,u\circ \gamma)<\eta$ for any $\gamma\in K$. This yields 
\[ \sup_{\gamma\in K} |f\circ u_{\varepsilon_i} (\gamma)- f\circ u (\gamma)| < \delta\] 
for large enough $i$,, which shows the claim. 

Conversely, assume that $Z^{\varepsilon,x}_t$ converges in law to $Y^x_t$ for any $x\in M$. As in the proof of Theorem \ref{t:main} we consider the processes $\bar Z^{\varepsilon,x}_t = u^{-1}(Z^{\varepsilon,x}_t)$, which are Brownian motions on $(M,(u^\varepsilon)^\sharp e)$ starting at $x$. Since the metric $h_e:=(u^\varepsilon)^\sharp e$ is smooth it holds 
\[\mathbb E[ f(\bar Z^{\varepsilon,x}_t)] = \int_M k_{h_\varepsilon}(t,x,y)f(y)d\mathrm{vol}_{h_\varepsilon}(y)\]
 for any $f\in C^\infty(M)$. By assumption, $\bar Z^{\varepsilon,x}_t$ converges in law to $u^{-1}(Y_t^x) = X^x_t$ which is Brownian motion on $(M,g)$. Consequently, 
 \[\lim_{\varepsilon\to 0 } \int_M k_{h_\varepsilon}(t,x,y)f(y)d\mathrm{vol}_{h_\varepsilon}(y) = \int_M k_g(t,x,y)f(y)d\mathrm{vol}_{g}(y)\,.\]
 On the other hand, $h_\varepsilon \to u^\sharp e$ in $C^{0,\alpha}$. We can therefore apply Lemma \ref{l:heatconvergence} to deduce that in fact 
 \begin{align*}\int_M k_g(t,x,y)f(y)d\mathrm{vol}_{g}(y)&= \lim_{\varepsilon\to 0 } \int_M k_{h_\varepsilon}(t,x,y)f(y)d\mathrm{vol}_{h_\varepsilon}(y) \\
 &= \int_M k_{u^\sharp e}(t,x,y)f(y)d\mathrm{vol}_{u^\sharp e}(y)\,\end{align*}
 for any $f\in C^\infty(M)$. As in the proof of Theorem \ref{t:main} we can localize and invoke Varadhan's lemma (more precisely the low-regularity version due to \cite{Norris}) to find $ d_g(x,y)=d_{u^\sharp e}(x,y)$ for all $x,y\in M$ close enough. We conclude by appealing to Lemma \ref{l:distanceiso}.

 \subsection{Markov property and Brownian motion on Riemannian manifolds with irregular metric}\label{s:Markovproperty} Let $(N,h)$ be a smooth closed Riemannian manifold with continuous metric $h$. Since $h$ is only continuous we cannot define the martingale problem for $\frac{1}{2}\Delta_h$ since (as observed in section \ref{s:heateq}) $\Delta_h$ is only well defined as a distribution. However, the Cauchy problem for the heat equation can be solved uniquely with the help of the heat kernel $k_h$. For $x\in N$ we define a Brownian motion on $(N,h)$ starting at $x$ to be a Markov process $\{X_t\}_{t\geq 0}$ on some filtered probability space $(\Omega, \mathcal F, \mathcal F_t,\mathbb P)$ with transition density given by the heat kernel $k_h$ and for which $\mathbb P( X_0 =x ) =1 $. 

 Now let $M$ be a smooth closed manifold and consider a $C^1$-regular embedding
$u:M\to \R^q$. We then have
 \begin{lemma} 
 \label{l:markov-bm} Assume $\{X_t\}_{t\geq 0}$ is a Brownian motion on $(M,u^\sharp e)$ starting at $x\in M$. Then $Y_t := u(X_t)$ is a Brownian motion on $(u(M), \iota^\sharp e)$.
 \end{lemma}

 \begin{proof}
Let $\Gamma\subset \Sigma$ be a Borel subset, $s>t\geq 0$. We need to show that 
\[\mathbb E[ \chi_\Gamma(Y_{t+s})\vert \mathcal F_t ] = \int_\Gamma k_{\iota^\sharp e} (s,Y_t,y)\,d\mathrm{vol}_{\iota^\sharp e}(y)\,.\]
This follows from the Markov property of $X_t$ and Lemma \ref{l:heatkernels}. Indeed, firstly, since $z\mapsto v(z,s) = \int_\Gamma k_{\iota^\sharp e} (s,z,y)\,d\mathrm{vol}_{\iota^\sharp e}(y) $ is continuous for any $s>0$ (see proof of Lemma \ref{l:heatconvergence}) and $Y_t= u(X_t)$, the right hand side, which is given by $v(u(X_t),s)$, is $\mathcal F_t$ measurable for any $s$. Now fix $A\in \mathcal F_t$. Then by Lemma \ref{l:heatkernels} and the Markov property of $X_t$ it follows 
\begin{align*}
&\int_A \left( \int_\Sigma k_{\iota^\sharp e} (s,Y_t,y)\chi_\Gamma(y)\,d\mathrm{vol}_{\iota^\sharp e}(y)\right) d\mathbb P \\
&\quad \quad \quad  =  \int_A \left( \int_M k_{u^\sharp e} (s,X_t,y)\chi_{u^{-1}(\Gamma)}(y)\,d\mathrm{vol}_{u^\sharp e}(y)\right)d\mathbb P \\
&\quad \quad \quad = \int_A \chi_{u^{-1}(\Gamma)} (X_{t+s})d\mathbb P \\
&\quad \quad \quad = \int_A \chi_\Gamma (Y_{t+s})d\mathbb P\,.\qedhere
\end{align*}

We combine Lemma~\ref{l:markov-bm}, Corollary \ref{c:markov1} and the first part of Theorem~\ref{t:renorm} to establish Corollary~\ref{cor:equivalence}. 
\end{proof}

%% file: thermo.tex
\section{Thermodynamics of isometric embedding}
\label{sec:thermo}
\subsection{Isometric embeddings, artificial intelligence and turbulence}
\label{subsec:embed}
The work presented in this paper is part of a program to construct Gibbs measures supported on isometric embeddings. Informally, the underlying questions are: `how do we construct typical isometric embeddings' and `what are their universal properties'?  

This program was first motivated by the embedding-turbulence analogy~\cite{DeS3}. The primary empirical reality in turbulence is the universality of the Kolmogorov spectrum. The Kolmogorov spectrum is a feature of the {\em statistical\/} theory of turbulence. By contrast, despite its origin, the Onsager conjecture (now theorem) has no explicit probabilistic content. Our work began as a randomization of Nash's scheme in order to bridge this divide. In this context, the construction of Gibbs measures allows critical exponents in PDE theory to be seen as the mathematical counterpart of critical exponent phenomena in condensed matter physics~\cite{goldenfeld2018lectures}. The isometric embedding problem also appears in quantum field theory under the guise of the nonlinear sigma model~\cite{Friedan1,Friedan2}. The physical context of these problems provides a rich source of inspiration. For example, the embedding-turbulence analogy allows us to conjecture universality for isometric embeddings and to seek models and numerical methods to test such conjectures. 

The isometric embedding problem also arises in several mathematical approaches to artificial intelligence. Diffusion maps, heat kernel embeddings, and geometric deep learning are successful frameworks in machine learning that implicitly use isometric embedding. This is a vast area; the papers~\cite{Belkin,lecun,Maggioni,Singer-Wu} provide an introduction to some of its scope. A foundational result in the area, due to Berard, Besson and Gallot~\cite{BBG}, is the construction of `almost' isometric embeddings $(M,g) \to L^2(M,g;\R)$ using heat kernels. Unlike Nash's embedding theorems, the heat kernel embeddings are approximate, not exact, and they map $M$ into an infinite-dimensional space, not $\R^q$. However, they are geometrically natural and admit robust numerical implementations. Thus, the use of heat kernel embeddings has stimulated new attempts to determine `canonical' isometric embeddings in $\R^q$~\cite{Portegies,Wang-Zhu}, extensions of the heat kernel method to RCD(K,N) spaces~\cite{Ambrosio,Huang2023isometric}, and computational relaxation schemes~\cite{Mcqueen}. Despite the extensive use of heat kernel embeddings, there appears to have been no attempt to use probabilistic methods, especially stochastic calculus, to rethink Nash's theorems prior to our work. 

Grenander's pattern theory provides a Bayesian framework for many problems in cognition (see~\cite{M-pt,Mumford-Desolneux} for introductions).  However, these Bayesian methods are not competitive with deep learning in computer vision, since they rely on the construction of geometric priors and fast optimization over parameters that determine the priors. Neither of these steps has been adequately resolved. Our construction of Gibbs measures for isometric embeddings is also stimulated by a desire to bridge the divide between Bayesian methods and deep learning.  

The above applications emphasize the need to balance rigorous results with fast algorithms, and to balance concept with technique, when re-examining Nash's work. The main conceptual insight formulated in~\cite{GM-gsi} is that the use of information theory provides unity between the many applications above. In the context of physics, we follow Jaynes and approach Gibbs measures from an information theoretic perspective~\cite{Jaynes}. We also augment the embedding-turbulence analogy by developing an embedding-RMT (for random matrix theory) analogy, so that the ties to mathematical physics are made explicit. The heart of the matter, however, is an information theoretic interpretation of embedding that goes roughly as follows. 

We view embedding as a form of information transfer between a source  and an observer. The process of information transfer is complete when all measurements of distances by the observer agree with those at the source. This viewpoint is both Bayesian and information theoretic. It places the emphasis not on the structure of the manifold, but on a more primitive aspect of the problem, the measurement of length. In the Bayesian interpretation, the world is random and both the source and the observer are stochastic processes with well-defined parameters. The process of successive approximation implict in Nash’s work can now be modeled as an optimal control strategy by an observer tuning a model in response to measurements of signals from the source. Thus, embedding is simply ‘replication’ and the process of replication is complete when all measurements by the observer and the source agree on a common set of questions (here it is the question: ‘what is the distance between points $x$ and $y$\,?’). 

The approach in this paper has been chosen to implement this information theoretic viewpoint in its simplest form.

\subsection{A gedanken experiment for the measurement of distance}
The power of the isometric embedding problem lies in its minimalism and generality. It is minimal because equation~\qref{eq:embed1} captures the equality of infinitesimal lengths in $(M,g)$ and on $u(M) \subset (\R^q,e)$. It is general because Nash's embedding theorems hold for all Riemannian manifolds $(M,g)$ under mild regularity and toplogical assumptions. The construction of Gibbs measures in mathematical physics is typically guided by an underlying dynamical system. Thus, in order to construct Gibbs measures supported on solutions to~\eqref{eq:embed1} it is necessary to use a dynamical system that respects the minimalism of the problem. 

The gedanken experiment that underlies Theorem~\ref{t:renorm} is this: How do we model the measurement of length for an embedding of $M$?  The ideal model must be consistent with the rigorous analysis of equation~\eqref{eq:embed1}, as well as its scientific applications.

Theorem~\ref{t:renorm} is formulated to capture the minimal thermodynamics of embedding in the following sense. It formalizes the idea that $u$ is isometric if and only if intrinsic and extrinsic observers measure exactly the same length between each pair of points. Of course, this is what the PDE~\eqref{eq:embed1} means implicitly. What is new is the idea that the use of Brownian motion on $(M,g)$ provides an explicit model for the measurement of distance by distinct explorers of the space $u(M)$. 

More precisely, the equivalence in law of the two stochastic processes $Y_t$ and $Z_t$, corresponds conceptually to the fact that intrinsic and extrinsic Brownian observers explore the space $\Sigma =u(M)$ in an equivalent manner. The reason we distinguish between $W_t \in \R^n$ and $B_t \in \R^q$ in the SDE's defining $U_t$ and $Z_t$ is that the use of distinct probability spaces to construct Brownian motion formalizes the notion of measurement of distance in different frames of reference. The heat kernel $k_\tau^Y(x,y)$ for the process $Y_t$ provides the best guess of the distance $d_g(x,y)$ between $x$ and $y$ as measured by an intrinsic observer at temperature $\tau$, with $d_g(x,y)^2 \approx -2\tau \log k_\tau^Y(x,y)$. Similarly, the heat kernel $k_\tau^Z$ provides an approximation to measurements by an extrinsic observer.  Thus, these two constructions of Brownian motion, along with the use of Varadhan's lemma, provide a minimal model for the measurement of distances by each observer. 
\subsection{Curvature, RLE and stochastic mechanics} 
Theorem~\ref{t:curvature} may also be interpreted along the lines of Nelson's stochastic mechanics~\cite{Dohrn-Guerra,Nelson}. The correction term $\tfrac{1}{2}H$ in equation~\eqref{eq:curvature} may be interpreted as a constraint that ensures that the process $Z_t$ does not leave $\Sigma$. It is the analog in stochastic mechanics  of the centripetal acceleration in Newtonian mechanics. The interpretation of (minus a half) mean curvature as a constraint force is an analog of Maxwell's derivation of the pressure on the boundary of a domain containing a hard-sphere gas using the time average of momentum transfer during collision. Unlike Maxwell, however, we do not postulate the existence of a hard sphere gas, only the existence of Wiener measure. Theorem~\ref{t:curvature} and (conjectural) equation~(\ref{eq:smg}) provide a rigorous microscopic explanation for the pressure, even in the low-regularity regime. 

An imprecise, but useful, caricature of this effect is as follows: subtract the drift from the Stratonovich equation~\eqref{e:extrinsicBM} to obtain the formal \Ito\/ SDE
\begin{equation}    
\label{eq:ito}
d\tilde{Z}_t = \sum_{a=1}^q P_a\left|_{\tilde Z_t} \right. \circ dB_t^a  - \frac{1}{2}H(\tilde Z_t) \,dt \stackrel{\mathrm{formally}}{=} \sum_{a=1}^q P_a\left|_{\tilde Z_t} \right. dB_t^a .
\end{equation}
The formal \Ito\/ equation is suggestive because it tells us that the fluctuations are
tangential. However, Theorem~\ref{t:curvature} tells us that $\tilde{Z}_t$ does not lie in $u(M)$: it is pushed outwards normally by the  \Ito\/ correction.

We have used this insight to provide a geometric construction of Dyson Brownian motion~\cite{HIM23}, to construct an analog of Dyson Brownian motion in the Siegel half-space~\cite{MY2}, to find Gibbs sampling algorithms for low-rank psd matrices ~\cite{Yu2,Yu1}, and to shed new light on deep learning and optimization~\cite{MY1}. When applied to the embedding problem itself, we see that we may replace Nash's discrete scheme outlined in Section~\ref{subsec:background} by a geometric stochastic flow of a short embedding that uses only tangential noise to `push normally outwards'. The more subtle issue, however, is to determine the most fundamental evolution equation of this nature. Here we may use random matrix theory as a guide.

In each of the above examples, symmetries simplify the analysis, allowing us to formulate a unifying Riemannian Langevin equation (RLE) of the form~\footnote{Equation~\eqref{e:RLE} is formal because, like equation~\eqref{eq:ito}, it is written in the \Ito\/ form, not the Stratonovich form. We have chosen this form so that the correspondence with the (Euclidean) Langevin equation is clear.} 
\begin{align}\label{e:RLE}
    dX_t=\mathrm{grad}_g S(X_t) + dB^{g,\beta}_t\,.
\end{align}
Here $X_t$ is a stochastic process on a Riemannian manifold $(M,g)$, $B^{g,\beta}_t$ is intrinsic Brownian motion on $(M,g)$ at inverse temperature $\beta$, and $S = \log W$ is a Boltzmann entropy such that the `number of microstates' $W$ is the volume of a group orbit. An additional potential may be included so that the entropy $S$ is replaced by minus a free energy; however, it is the understanding of $S$ that is most fundamental in each of these problems.

The stochastic evolution used to establish the existence of Gibbs measures for the isometric embedding problem and turbulence must adhere strictly to thermodynamic principles. The above examples reveal clearly that the RLE for stochastic gradient ascent of entropy is the most fundamental such model. This leads us to seek natural infinite-dimensional group orbits corresponding to the isometric embedding problem, defining their entropy $S$ by replacing volume with a Fredholm determinant. The theory of stochastic flows provides a powerful tool for this purpose. But as in this paper, this  viewpoint leads us back to classical questions in stochastic flows and random matrix theory, allowing us to see them in a new light. 

The fundamental problem within this class is as follows. Given $(M,g)$ our task is to determine the low-regularity renormalization of Brownian motion in the diffeomorphism group $\mathrm{Diff}(M)$, extending work of Baxendale and Kunita~\cite{Baxendale-diff,Kunita}. The renormalization involves resolving a classical degeneracy: there are many Gaussian sections of $T\mathcal{M}$ whose stochastic flows have one-point marginals that are Brownian motion~\cite[Ch.4]{Kunita}. That is, there are {\em many\/} intrinsic constructions of Brownian motion on $(\mathcal{M},g)$. The fundamental new idea that is suggested by the embedding-RMT analogy is that a principled renormalization is {\em not\/} the Eells, Elworthy, Malliavin construction. Instead, we approach Brownian motion in  $\mathrm{Diff}(M)$ using {\em stochastic gradient flows\/} given by Stratonovich SDE
\begin{equation}
\label{eq:renorm-flow}
dX_t = \mathrm{grad}_g (\circ \, d\psi_t),
\end{equation}
for a stationary random field $\psi_t: \mathcal{M}\to \R$. The law of $\psi$ is determined by a {\em stochastic\/} Nash lemma, replacing~\cite[Lemma 1]{Nash1} with an infinite-dimensional matrix-completion problem. 

In this way, we see a subtle persistence of Nash's ideas, even when the isometric embedding problem is studied with entirely probabilistic techniques. These RLE will be explained at greater length in forthcoming work by the authors.

Taken together, our theorems and gedanken experiment, express the idea that the true nature of the isometric embedding problem lies in understanding the implicit role of measurement, and gauge invariance, in the character of physical law.